%% file: paper.tex
\documentclass[12pt]{amsart}

\usepackage[utf8x]{inputenc}
\usepackage[english]{babel}
\usepackage{amsmath, amsfonts, amssymb}
\usepackage{color}
\usepackage{tikz}
\usetikzlibrary{positioning,shapes,shadows,arrows}
\usepackage{enumitem}
\usepackage{hhline}
\usepackage{array}
\usepackage[colorlinks=true,linkcolor=blue,citecolor=blue,urlcolor=blue]{hyperref}
\usepackage{cite}
\usepackage{listings}
\usepackage{shuffle}
\usepackage{todonotes}
\usepackage{frcursive}
\usepackage{mathtools}
\usepackage{pdflscape}
\usepackage{amsaddr}

\title{The Rise-Contact involution on Tamari intervals}

\author{Viviane Pons}

\address{LRI, Univ. Paris-Sud - CNRS - Centrale Supelec - Univ. Paris-Saclay}
\email{viviane.pons@lri.fr}

\numberwithin{equation}{section}
\newtheorem{Theorem}{Theorem}[section]

\newtheorem{Proposition}[Theorem]{Proposition}
\newtheorem{Lemma}[Theorem]{Lemma}
\newtheorem{Definition}[Theorem]{Definition}

\newtheorem{Remark}[Theorem]{Remark}

\newtheorem{proofclaim}{Claim}
\makeatletter
\@addtoreset{proofclaim}{Theorem}
\makeatother

\newcommand{\NN}{\mathbb{N}}

\newcommand{\Tamk}[1]{\mathcal{T}_{#1}} 
\newcommand{\Tamnm}{\mathcal{T}_{n}^{(m)}} 
\newcommand{\Tam}[2]{\mathcal{T}_{#1}^{(#2)}}
\newcommand{\Tamntm}{\Tamk{n \times m}}

\newcommand{\trprec}{\vartriangleleft} 
\newcommand{\trsucc}{\vartriangleright} 
\newcommand{\trpreceq}{\trianglelefteq} 
\newcommand{\trsucceq}{\trianglerighteq} 
\newcommand{\ntrprec}{\ntriangleleft}

\newcommand\dec{F_{\ge}} 
\newcommand\inc{F_{\le}} 

\DeclareMathOperator{\rel}{rel} 
\DeclareMathOperator{\tree}{tree} 
\newcommand{\pleft}{~\vec{\bullet}~} 
\newcommand{\pright}[1]{~\overleftarrow{\delta_{#1}}~} 

\DeclareMathOperator{\contacts}{c_0} 
\DeclareMathOperator{\rises}{r_0} 
\DeclareMathOperator{\contactsV}{\mathbf{C}} 
\DeclareMathOperator{\risesV}{\mathbf{R}} 
\DeclareMathOperator{\mcontacts}{{c_{(m)}}_0} 
\DeclareMathOperator{\mrises}{{r_{(m)}}_0} 
\DeclareMathOperator{\mcontactsV}{\mathbf{C}_m} 
\DeclareMathOperator{\mrisesV}{\mathbf{R}_m} 
\DeclareMathOperator{\DC}{\mathbf{DC}} 
\DeclareMathOperator{\IC}{\mathbf{IC}} 
\DeclareMathOperator{\contactsP}{\mathcal{C}} 
\DeclareMathOperator{\risesP}{\mathcal{R}} 
\DeclareMathOperator{\mcontactsP}{\mathcal{C}_m} 
\DeclareMathOperator{\mrisesP}{\mathcal{R}_m} 
\DeclareMathOperator{\graftingTree}{\Delta} 
\DeclareMathOperator{\leftbranch}{\phi} 
\DeclareMathOperator{\compl}{\psi} 
\DeclareMathOperator{\risecontact}{\beta} 
\DeclareMathOperator{\mrisecontact}{\beta_m} 
\DeclareMathOperator{\TInv}{TInv} 
\DeclareMathOperator{\labels}{labels} 
\DeclareMathOperator{\DD}{\textbf{D}} 
\DeclareMathOperator{\expand}{expand} 
\DeclareMathOperator{\contract}{contract} 
\newcommand{\contactsStep}[1]{\operatorname{c}_{#1}} 
\newcommand{\risesStep}[1]{\operatorname{r}_{#1}} 
\newcommand{\mcontactsStep}[1]{{\operatorname{c}_{(m)}}_{#1}} 
\newcommand{\mrisesStep}[1]{{\operatorname{r}_{(m)}}_{#1}} 
\newcommand{\dcstep}[1]{\operatorname{dc}_{#1}} 
\newcommand{\icstep}[1]{\operatorname{ic}_{#1}} 
\newcommand{\icinf}{\operatorname{ic}_{\infty}} 
\newcommand{\tI}{\tilde{I}} 

\DeclareMathOperator{\distance}{d} 
\DeclareMathOperator{\size}{size} 

\definecolor{darkGreen}{RGB}{23,103,1}

\newcommand{\red}[1]{\textbf{\textcolor{red}{#1}}}

\newcommand{\blue}[1]{\textcolor{blue}{#1}}
\newcommand{\green}[1]{\textcolor{darkGreen}{#1}}

\tikzstyle{Red} = [color = red]
\tikzstyle{Blue} = [color = blue]
\tikzstyle{Green} = [color = darkGreen]
\tikzstyle{Gray} = [color = gray]

\tikzstyle{Path} = [line width = 1.2]
\tikzstyle{StrongPath} =  [line width=2]
\tikzstyle{DPoint} = [fill, radius=0.1]
\tikzstyle{Line1} = [dashed]
\tikzstyle{Line2} = [dotted, ultra thick]

\tikzstyle{Point} = [fill, radius=0.08]
\tikzstyle{RedPoint} = [color = red, fill, radius=0.08]
\tikzstyle{BluePoint} = [color = blue, fill, radius=0.08]
\tikzstyle{GreenPoint} = [color = darkGreen, fill, radius=0.08]
\tikzstyle{RedPath} = [color = red]
\tikzstyle{BluePath} = [color = blue]
\tikzstyle{GreenPath} = [color = darkGreen]
\tikzstyle{GrayPath} = [color = gray]
\tikzstyle{StrongPath} =  [line width=2.5]
\tikzstyle{StrongPath2} =  [line width=1.5]
\tikzstyle{Leaf} = [color = gray]
\tikzstyle{RedLabel} = [color = red]
\tikzstyle{BlueLabel} = [color = blue]
\tikzstyle{GreenLabel} = [color = darkGreen]
\tikzstyle{Label} = [color = black]

\tikzstyle{Tree} = [draw, isosceles triangle, shape border rotate = 90, isosceles triangle apex angle = 60]

\begin{document}

\maketitle

\begin{abstract}
We describe an involution on Tamari intervals and $m$-Tamari intervals. This involution switches two sets of statistics known as the ``rises'' and the ``contacts'' and so proves an open conjecture from Préville-Ratelle on intervals of the $m$-Tamari lattice.
\end{abstract}

\subsubsection*{Acknowledgements} 
The author would like to thank Darij Grinberg for his time spent understanding interval-posets and the nice and constructive discussions that followed, which had a positive impact on this paper. She also thanks Grégory Châtel and Frédéric Chapoton for the original work on the interval-posets and bijections which later on led to this result.

Finally, the computation and tests needed along the research were done using the open-source mathematical software {\tt SageMath}~\cite{SageMath2017} and its combinatorics features developed by the {\tt Sage-Combinat} community~\cite{SAGE_COMBINAT}. The recent development funded by the OpenDreamKit Horizon 2020 European Research Infrastructures project (\#676541) helped providing the live environment~\cite{JNotebook} which complements this paper.

\section{Introduction}

The Tamari lattice \cite{Tamari1, Tamari2} is a well known lattice on Catalan objects, most frequently described on binary trees, Dyck paths, and triangulations of a polygon. Among its many interesting combinatorial properties, we find the study of its intervals. Indeed, it was shown by Chapoton \cite{Chap} that the number of intervals of the Tamari lattice on objects of size $n$ is given by 

\begin{equation}
\label{eq:intervals-formula}
\frac{2}{n(n +1)} \binom{4 n + 1}{n - 1}.
\end{equation}

This is a surprising result. Indeed, it is not common that we find a closed formula counting intervals in a lattice. For example, there is no such formula to count the intervals of the weak order on permutations. Even more surprising is that this formula also counts the number of simple rooted triangular maps, which led Bernardi and Bonichon to describe a bijection between Tamari intervals and said maps \cite{BijTriangulations}. This is a strong indication that Tamari intervals have deep and interesting combinatorial properties.

One generalization of the Tamari lattice is to describe it on $m$-Catalan objects. This was done by Bergeron and Préville-Ratelle \cite{BergmTamari}. Again, they conjectured that the number of intervals could be counted by a closed formula, which was later proved in \cite{mTamari}:

\begin{equation}
\label{eq:m-intervals-formula}
\frac{m+1}{n(mn +1)} \binom{(m+1)^2 n + m}{n - 1}.
\end{equation} 

In this case, the connection to maps is still an open question. The rich combinatorics of Tamari intervals and their generalizations has led to a surge of effort in their study. This is motivated by their connections with various subjects such as algebra, representation theory, maps, and more.
For example, in~\cite{PRE_RognerudModern}, the author motivates the study of some subfamilies of intervals by connections to operads theory as well as path algebras. Another fundamental example is the work of Begeron and Préville-Ratelle on diagonal harmonic polynomials~\cite{BergmTamari} which has led to the study of $m$-Tamari lattices and more recently generalized Tamari lattices~\cite{RatelleViennot, FangPrevilleRatelle}. The relation to maps, and more specifically Schnyder woods~\cite{BijTriangulations} is a motivation for studying the relation between Tamari intervals and certain types of decorated trees (see for example~\cite{WeylChamber} and~\cite{FangBeta10}). A by-product of our paper is to introduce a new family of trees, the \emph{grafting trees}, which are very close to these decorated trees. In fact, they are in bijection with $(1,1)$ decoration trees of~\cite{CoriSchaefferDescTrees}. 

The goal of the present paper is to prove a certain equi-distribution of statistics on Tamari intervals related to \emph{contacts} and \emph{rises} of the involved Dyck paths. This was first noticed in \cite{mTamari}. At this stage, the equi-distribution could be seen directly on the generating function of the intervals but there was no combinatorial explanation. In his thesis \cite{PRThesis}, Préville-Ratelle developed the subject and left some open problems and conjectures. The one related to the contacts and rises of Tamari intervals is Conjecture~17, which we propose to prove in this paper. It describes an equi-distribution not only between two statistics (as in \cite{mTamari}) but between two sets of statistics. Basically, in \cite{mTamari}, only the initial rise of a Dyck path was considered, whereas in Conjecture~17, Préville-Ratelle considers all positive rises of the Dyck path. Besides, a third statistic is described, the \emph{distance}, which also appears in many other open conjectures and problems of Préville-Ratelle 's thesis: it is related to trivariate diagonal harmonics, which is the original motivation of the $m$-Tamari lattice. According to Préville-Ratelle, Conjecture~17 can be proved both combinatorially\footnote{Gilles Schaeffer says that this derives from a natural involution on maps.} and through the generating function when~$m=1$. But until now, there was no proof of this result when~$m > 1$.

To prove this conjecture, we use some combinatorial objects that we introduced in a previous paper on Tamari intervals \cite{IntervalPosetsInitial}: the interval-posets. They are posets on integers, satisfying some simple local rules, and are in bijections with the Tamari intervals. Besides, their structure includes two planar forests (from the two bounds of the Tamari interval), which are very similar to the Schnyder woods of the triangular planar maps. Another quality of interval-posets is that $m$-Tamari intervals are also in bijection with a sub-family of interval-posets, which was the key to prove the result when~$m > 1$.

Section~\ref{sec:interval-posets} of this paper gives a proper definition of Tamari interval-posets and re-explores the link with the Tamari lattice in the context of our problem. In Section~\ref{sec:statistics}, we describe the \emph{rise}, \emph{contact}, and \emph{distance} statistics and their relations to interval-poset statistics. This allows us to state Theorem~\ref{thm:main-result-classical}, which expresses our version of Conjecture~17 in the case $m=1$. Section~\ref{sec:involutions} is dedicated to the proof of Theorem~\ref{thm:main-result-classical} through an involution on interval-posets described in Theorem~\ref{thm:rise-contact-statistics}. However, the main results of our paper lie in our last section, Section~\ref{sec:mtam}, where we are able to generalize the involution to the $m > 1$ case. Theorem~\ref{thm:main-result-general} is a direct reformulation of Conjecture~17 from \cite{PRThesis}. It is a consequence of Theorem~\ref{thm:m-rise-contact-involution}, which describes an involution on intervals of the $m$-Tamari lattice.

\begin{Remark}
A previous version of this involution was described in an extended abstract~\cite{ME_FPSAC2014}. This was only for the $m=1$ case and did not include the whole set of statistics. Also, in this original description, the fact that it was an involution could be proved but was not clear. We leave it to the curious reader to see that the bijection described in~\cite{ME_FPSAC2014} is indeed the same as the one we are presenting in details now.
\end{Remark}

\begin{Remark}
This paper comes with a complement {\tt SageMath-Jupyter} notebook~\cite{JNotebook} available on {\tt github} and {\tt binder}. This notebook contains {\tt SageMath} code for all computations and algorithms described in the paper. The {\tt binder} system allows the reader to run and edit the notebook online.
\end{Remark}

\section{Tamari Interval-posets}
\label{sec:interval-posets}

\subsection{Definition}

Let us first introduce some notations that we will need further on. In the following, if $P$ is a poset, then we denote by $\trprec_{P}$,
$\trpreceq_{P}$, $\trsucc_{P}$ and $\trsucceq_{P}$ the smaller, smaller-or-equal,
greater and greater-or-equal, respectively, relations of the poset $P$. When
the poset $P$ can be uniquely inferred from the context, we will sometimes
leave out the subscript ``$P$''. We write
\begin{equation}
\rel(P) = \lbrace (x,y) \in P, x \trprec y \rbrace
\end{equation}
for the set of relations of $P$. A relation $(x,y)$ is said to be a \emph{cover relation} if there is no $z$ in $P$ such that $x \trprec z \trprec y$. The Hasse diagram of a poset $P$ is the directed graph formed by the cover relations of the poset. A poset is traditionally represented by its Hasse diagram.

We say that we \emph{add} a relation $(i,j)$ to a poset $P$ when we add $(i,j)$ to $\rel(P)$ along with all relations obtained by transitivity (this requires that neither $i \trprec_P j$ nor $j \trprec_P i$ before the addition). Basically, this means we add an edge to the Hasse Diagram. The new poset $P$ is then an \emph{extension} of the original poset.

We now give a first possible definition of interval-posets.

\begin{Definition}
\label{def:interval-poset}
A \emph{Tamari interval-poset} (simply referred as \emph{interval-poset} in this paper) is a poset $P$ on $\left\{  1,2,...,n\right\}  $ for some $n\in\NN$, such that all triplets $a < b < c$ in $P$
 satisfy the following property, which we call the \emph{Tamari axiom}:

\begin{itemize}
\item $a \trprec c$ implies $b \trprec c$;
\item $c \trprec a$ implies $b \trprec a$.
\end{itemize}
\end{Definition}

Figure \ref{fig:interval-poset-example} shows an example and a counter-example of interval-posets.
The first poset is indeed an interval-poset. The Tamari axiom has to be checked on every $a < b < c$ such 
that there is a relation between $a$ and $c$: we check the axiom on $1 < 2 < 3$ and $3 < 4 < 5$ and it is satisfied.
The second poset of Figure \ref{fig:interval-poset-example} is not an interval poset: it contains 
$1 \trprec 3$ but not $2 \trprec 3$ so the Tamari axiom is not satisfied for $1 < 2 < 3$.

\begin{figure}[ht]
\input{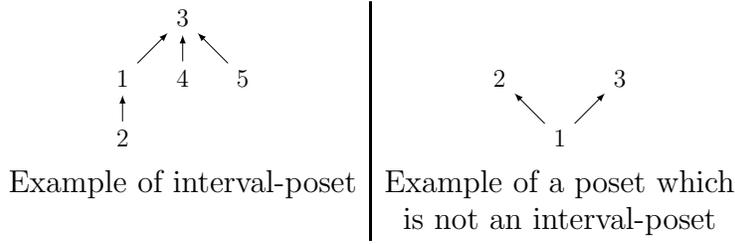}
\caption{Example and counter-example of interval-poset}
\label{fig:interval-poset-example}
\end{figure}

\begin{Definition}
Let $P$ be an interval-poset and $a,b \in P$ such that $a < b$. Then

\begin{itemize}
\item if $a \trprec b$, then $(a,b)$ is said to be an \emph{increasing relation} of $P$.

\item if $b \trprec a$, then $(b,a)$ is said to be a \emph{decreasing relation} of $P$.
\end{itemize}
\end{Definition}

As an example, the increasing relations of the interval-poset of Figure~\ref{fig:interval-poset-example} 
are $(1,3)$ and $(2,3)$ and the decreasing relations are $(2,1)$, $(4,3)$, and $(5,3)$. Clearly a relation $x \trprec y$ 
is always either increasing or decreasing and so one can split the relations of $P$ into two non-intersecting sets. 

\begin{Definition}
Let $P$ be an interval-poset. Then, the \emph{final forest} of $P$, denoted by $\dec(P)$, is the poset formed by the decreasing relations of $P$, \emph{i.e.}, $b \trprec_{\dec(P)} a$ if and only if $(b,a)$ is a decreasing relation of $P$. Similarly,
the \emph{initial forest} of $P$, denoted by $\inc(P)$, is the poset formed by the increasing relations of $P$.
\end{Definition}

By Definition \ref{def:interval-poset} it is immediate that the final and initial forests of an interval-poset are also interval-posets. 
By extension, we say that an interval-poset containing only decreasing (resp. increasing) relations is a final forest (resp. initial
forest). The designation \emph{forest} comes from the result proved in \cite{IntervalPosetsInitial} that an interval-poset containing only increasing (resp. decreasing) relations has indeed the structure of a planar forest, \emph{i.e.}, every vertex in the Hasse diagram has at most one outgoing edge.

The increasing and decreasing relations of an interval-poset play a significant role in the structure and properties of the object. We thus follow the convention described in \cite{IntervalPosetsInitial} to draw interval-posets, which differs from the usual representation of posets through their Hasse diagram. Indeed, each interval-poset is represented with an overlay of the Hasse Diagrams of both its initial and final forests. By convention, an increasing relation $b \trprec c$ with $b < c$ is represented in blue with $c$ on the right of $b$. A decreasing relation $b \trprec a$ with $a < b$ is represented in red with $a$ above $b$. In general a relation (either increasing or decreasing) between two vertices $x \trprec y$ is always represented such that $y$ is on a righter and upper position compared to $x$. Thus, the color code, even though practical, is not essential to read the figures. Figure~\ref{fig:interval-poset-forests} shows the final and initial forests of the interval-poset of Figure~\ref{fig:interval-poset-example}. A more comprehensive example is shown in Figure \ref{fig:interval-poset-example2}. Following our conventions, you can read off, for example, that $3 \trprec 4 \trprec 5$ and that $9 \trprec 8 \trprec 5$.

\begin{figure}[ht]
\begin{center}
\input{figures/interval-poset-forests}
\end{center}
\caption{Final and initial forests of an interval-poset}
\label{fig:interval-poset-forests}
\end{figure}

\begin{figure}[ht]
\begin{center}
\scalebox{0.8}{
\input{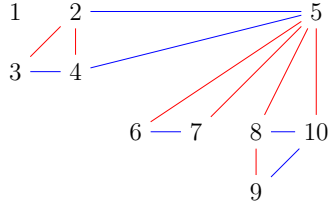}
}
\end{center}
\caption{An example of an interval-poset}
\label{fig:interval-poset-example2}
\end{figure}

We also define some vocabulary on the vertices of the interval-posets related to the initial and final forests.
\begin{Definition}
Let $P$ be an interval-poset. Then
\begin{itemize}
\item $b$ is said to be \emph{a decreasing root} of $P$ if there is no $a < b$ with the decreasing relation $b \trprec a$;
\item $b$ is said to be \emph{an increasing root} of $P$ if there is no $c > b$ with the increasing relation $b \trprec c$;
\item an increasing-cover (resp. decreasing-cover) relation is a cover relation of the initial (resp. final) forest of $P$;
\item the \emph{decreasing children} of $b$ are all elements $c > b$ such that $c \trprec b$ is a decreasing-cover relation;
\item the \emph{increasing children} of $b$ are all elements $a < b$ such that $a \trprec b$ is an increasing-cover relation.
\end{itemize}
\end{Definition}
As an example, in Figure~\ref{fig:interval-poset-example2}: the decreasing roots are $1,2,5$, the increasing roots are $1,5,7,10$, there are 7 decreasing-cover relations (red edges) and 6 increasing-cover relations (blue edges), the decreasing children of 5 are $6, 7, 8, 10$ and its increasing children are 2 and 4.

We also need to refine the notion of extension related to increasing and decreasing relations.

\begin{Definition}
\label{def:interval-poset-extensions}
Let $I$ and $J$ be two interval-posets, we say that
\begin{itemize}
\item $J$ is an \emph{extension} of $I$ if for all $i,j$ in $I$, $i \trprec_I j$ implies $i \trprec_{J} j$;
\item $J$ is a \emph{decreasing-extension} of $I$ if $J$ is an extension of $I$ and for all $i,j$ such that $i \trprec_J j$ and $i \ntrprec_I j$ then $i > j$;
\item $J$ is an \emph{increasing-extension} of $I$ if $J$ is an extension of $I$ and for all $i,j$ such that $i \trprec_J j$ and $i \ntrprec_I j$ then $i < j$;
\end{itemize}
\end{Definition}

In other words, $J$ is an extension of $I$ if it is obtained by adding relations to $I$, it is a decreasing-extension if it is obtained by adding only decreasing relations and it is an increasing-extension if it is obtained by adding only increasing relations.

\begin{Remark}
\label{rem:adding-decreasing}
If you add a decreasing relation $(b,a)$ to an interval-poset $I$, all extra relations that are obtained by transitivity are also decreasing. Indeed, suppose that $J$ is obtained from $I$ by adding the relation $b \trprec a$ with $a < b$ (in particular neither $(a,b)$ nor $(b,a)$ is a relation of $I$). And suppose that the relation $i \trprec_J j$ with $i < j$ is added by transitivity, which means $i \ntrprec_I j$, $i \trpreceq_I b$ and $a \trpreceq_I j$. If $i < a$, the Tamari axiom on $(i,a,b)$ implies $a \trprec_I b$, which contradicts our initial statement. So we have $a < i < j$ and $a \trprec_I j$, the Tamari axiom on $(a,i,j)$ implies $i \trprec_I j$ and again contradicts our statement. Note on the other hand that nothing guarantees that the obtained poset is still an interval-poset.
Similarly, if you add an increasing relation $(a,b)$ to an interval-poset, you obtain an increasing-extension.
\end{Remark}

\subsection{The Tamari lattice}
\label{sec:tamari}

It was shown in \cite{IntervalPosetsInitial} that Tamari interval-posets are in bijection with intervals of the Tamari lattice. The main purpose of this paper is to prove a conjecture of Préville-Ratelle \cite{PRThesis} on Tamari intervals. To do so, we first give a detailed description of the relations between interval-posets and the realizations of the Tamari lattice in terms of trees and Dyck paths. Let us start with some reminder on the Tamari lattice. 

\begin{Definition}
A binary tree is recursively defined by being either
\begin{itemize}
\item the empty tree, denoted by $\emptyset$,
\item a pair of binary trees, respectively called \emph{left} and \emph{right} subtrees, grafted on a node.
\end{itemize}

If $L$ and $R$ are two binary trees, we denote by $\bullet (L,R)$ the binary tree obtained from $L$ and $R$ grafted on a node. 
\end{Definition}


What we call a binary tree is often called a \emph{planar binary tree} in the literature (as the order on the subtrees is important). Note that in our representation of binary trees, we never draw the empty subtrees.

The \emph{size} of a binary tree is defined recursively: the size of the empty tree is $0$, and the size of a tree $\bullet (L, R)$ is the sum of the sizes of $L$ and $R$ plus 1. It is also the number of nodes. For example, the following tree 
\scalebox{0.5}{\input{figures/trees/T3-2}}
has size 3, it is given by the recursive grafting $\bullet ( \bullet(\emptyset, \bullet(\emptyset, \emptyset) ) , \emptyset)$. It is well known that the unlabeled binary trees of size $n$ are counted by the $n^{th}$ Catalan number
\begin{equation}
\frac{1}{n+1}\binom{2n}{n}.
\end{equation}

\begin{Definition}[Standard binary search tree labeling]
Let $T$ be a binary tree of size $n$. The \emph{binary search tree labeling} of $T$ is the unique labeling of $T$ with labels $1, \dots, n$ such that for a node labeled $k$, all nodes on the left subtree of $k$ have labels smaller than $k$ and all nodes on the right subtree of $k$ have labels greater than $k$. An example is given in Figure \ref{fig:bst-example}.
\end{Definition}
\begin{figure}[ht]
\input{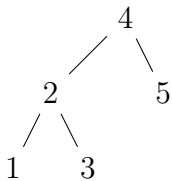}
\caption{A binary search tree labeling}
\label{fig:bst-example}
\end{figure}

In other words, the binary search tree labeling of $T$ is an in-order recursive traversal of $T$: left, root, right. For the rest of the paper, we identify binary trees with their corresponding binary search tree labeling. In particular, we write $v_1, \dots, v_n$ the nodes of $T$: the index of the node corresponds to its label in the binary search tree labeling.

To define the Tamari lattice, we need the following operation on binary trees.

\begin{Definition}
\label{def:tree-rotation}
Let $v_y$ be a node of $T$ with a non-empty left subtree of root $v_x$. The \emph{right rotation} of $T$ on $v_y$ is a local rewriting which follows Figure~\ref{fig:tree-right-rotation}, that is replacing $v_y( v_x(A,B), C)$ by $v_x(A,v_y(B,C))$ (note that $A$, $B$, or $C$ might be empty).

\begin{figure}[ht]
\centering

\input{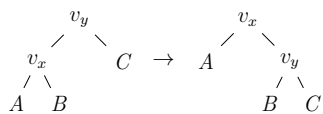}

\caption{Right rotation on a binary tree.}

\label{fig:tree-right-rotation}

\end{figure}
\end{Definition}

It is easy to check that the right rotation preserves the binary search tree labeling. It is the cover relation of the Tamari lattice \cite{Tamari1,Tamari2}: a binary tree $T$ is said to be bigger in the Tamari lattice than a binary tree $T'$ if it can be obtained from $T'$ through a sequence of right rotations. The lattices for the sizes 3 and 4 are given in Figure~\ref{fig:tamari-trees}.

\begin{figure}[ht]
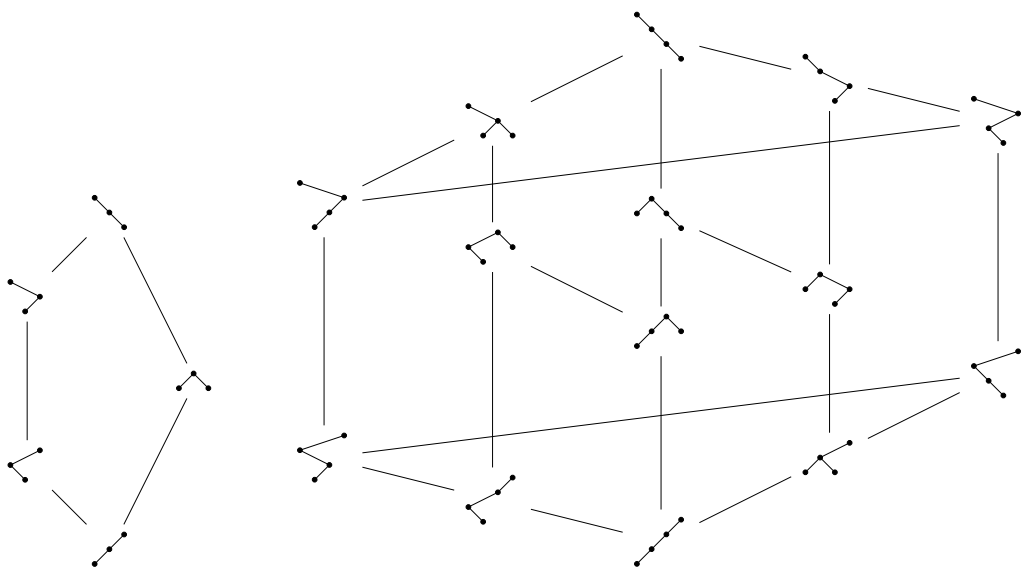

  
\hspace*{-1.5cm}
    \begin{tabular}{cc}
    \scalebox{0.7}{\input{figures/tamari_trees-3}}&
    \scalebox{0.7}{\input{figures/tamari_trees-4}}
    \end{tabular}
    
    \caption{Tamari lattice of sizes 3 and 4 on binary trees.}
    
    \label{fig:tamari-trees}

\end{figure}

Dyck paths are another common set of objects used to define the Tamari lattice. First, we recall their definition.
\begin{Definition}
A \emph{Dyck path} of size $n$ is a lattice path from the origin $(0,0)$ to the point $(2n,0)$ made from a sequence of \emph{up-steps} (steps of the form $(x, y) \to (x+1, y+1)$) and \emph{down-steps} (steps of the form $(x, y) \to (x+1, y-1)$) such that the path stays above the line $y=0$. 
\end{Definition}
A Dyck path can also be considered as a binary word by replacing up-steps by the letter $1$ and down-steps by $0$. We call a Dyck path \emph{primitive} if it only touches the line $y=0$ on its end points. As widely known, Dyck paths are also counted by the Catalan numbers. 
There are many ways to define a bijection between Dyck paths and binary trees. The one we use here is the only one which is consistent with the usual definition of the Tamari order on Dyck paths. 

\begin{Definition}
\label{def:dyck-tree}
We define the $\tree$ map from the set of all Dyck paths to the set of binary trees recursively. Let $D$ be a Dyck path. 
\begin{itemize}
\item If $D$ is empty, then $\tree(D)$ is the empty binary tree.
\item If $D$ is of size $n > 0$, then the binary word of $D$ can be written uniquely as $D_1 1 D_2 0$ where $D_1$ and $D_2$ are Dyck paths 
of size smaller than $n$ (in particular, they can be empty paths). Then $\tree(D)$ is the tree $\bullet (\tree(D_1), \tree(D_2))$.
\end{itemize}

\end{Definition}

Note that the path defined by $1D_2 0$ is primitive; it is the only non-empty right factor of the binary word of $D$ which is a primitive Dyck path.  Similarly, the subpath $D_1$ corresponds to the left factor of $D$ up to the last touching point before the end. Consequently, if $D$ is primitive, then $D = 1D_2 0$, while $D_1$ is empty and thus $\tree(D)$ is a binary tree whose left subtree is empty. If both $D_1$ and $D_2$ are empty, then $D =10$, the only Dyck path of size $1$, and $\tree(D)$ is the binary tree formed by a single node.

The $\tree$ map is a bijection and preserves the size as it is illustrated in Figure~\ref{fig:dyck-tree}.
\begin{figure}[ht]
\centering
\scalebox{0.8}{
\input{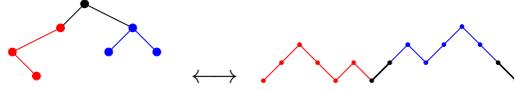}
}
    \caption{Bijection between Dyck paths and binary trees.}
    
    \label{fig:dyck-tree}
    
\end{figure}

Following this bijection, one can check that the right rotation on binary trees corresponds to the following operation on Dyck paths.
\begin{Definition}
 A \emph{right rotation} of a Dyck path $D$ consists of switching a down step $d$ followed by an up step with the primitive Dyck path starting right after $d$. (See Figure \ref{fig:rot-dyck}.)

\begin{figure}[ht]
\scalebox{0.8}{
\input{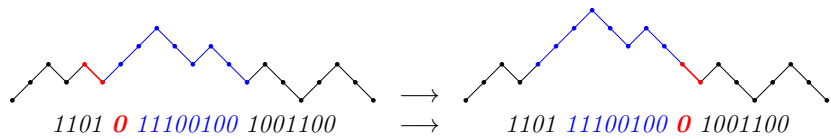}
}
\caption{Rotation on Dyck Paths.}
\label{fig:rot-dyck}
\end{figure}
\end{Definition}

By extension, we then say that a Dyck path $D$ is bigger than a Dyck path $D'$ in the Tamari lattice if it can be obtained from $D'$ through a series of right rotations. The Tamari lattices of sizes 3 and 4 in terms of Dyck paths are given in Figure~\ref{fig:tamari-dyck}.

\begin{figure}[ht]
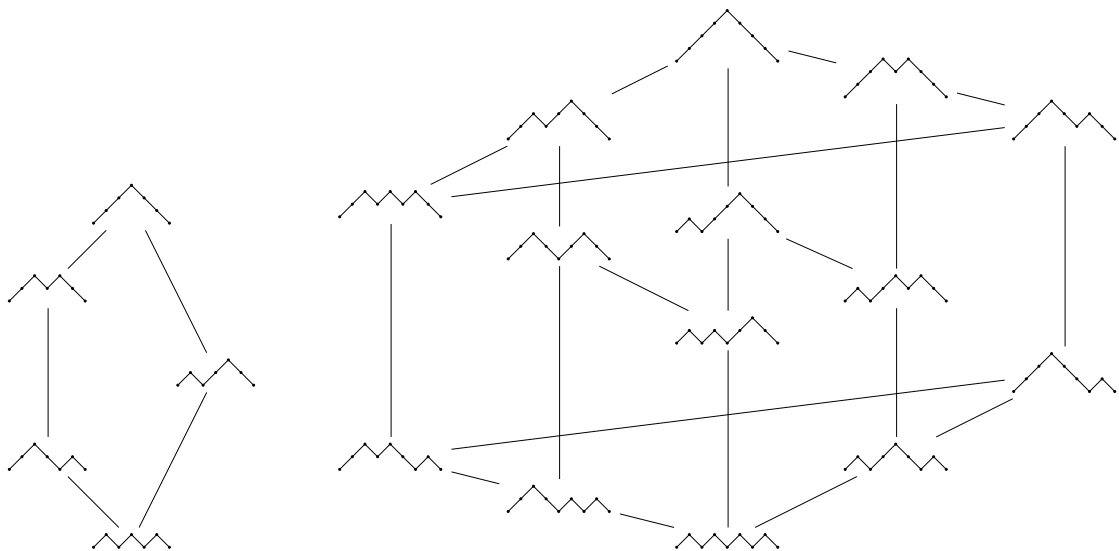

  
\hspace*{-1.5cm}
    \begin{tabular}{cc}
    \scalebox{0.7}{\input{figures/tamari_dyck-3}}&
    \scalebox{0.7}{\input{figures/tamari_dyck-4}}
    \end{tabular}
    
    \caption{Tamari lattices of sizes 3 and 4 on Dyck paths.}
    
    \label{fig:tamari-dyck}

\end{figure}

\subsection{Planar forests}

The bijection between interval-posets and intervals of the Tamari lattice uses a classical bijection between binary trees and planar forests. 


\begin{Definition}
Let $T$ be a binary tree of size $n$ and $v_1, \dots v_n$ its nodes taken in in-order as to follow the binary search tree labeling of~$T$. 

The final forest of $T$, $\dec(T)$ is the poset on $\left\{ 1, \dots, n \right\}$ whose relations are defined as follows: $b \trprec a$ if and only if $v_b$ is in the right subtree of~$v_a$. (Thus, $b \trprec a$ implies $b > a$.)

Similarly, the initial forest of $T$, $\inc(T)$, is the poset on $\left\{ 1, \dots, n \right\}$ whose relations are defined as follows: $a \trprec b$ if and only if $v_a$ is in the left subtree of $v_b$. (Thus, $a \trprec b$ implies $b > a$.)
\end{Definition}

\begin{figure}[ht]
\input{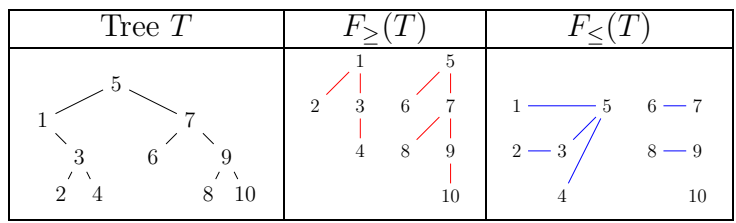}
\caption{A binary tree with its corresponding final and initial forests.}
\label{fig:forests}
\end{figure}

An example of the construction is given in Figure~\ref{fig:forests}. As explained in~\cite{IntervalPosetsInitial}, both the initial and the final forest constructions give bijections between binary trees and planar forests, \emph{i.e.}, forests of trees where the order on the trees is fixed as well as the orders of the subtrees of each node. Indeed, we first notice that the labeling on both images $\dec(T)$ and $\inc(T)$ is entirely canonical (such as the labeling on the binary tree) and can be retrieved by only fixing the order in which to read the trees and subtrees. Then these are actually well known bijections. The one giving the final forest is often referred to as ``left child = left brother'' because it can be achieved directly on the unlabeled binary tree by transforming every left child node into a left brother and by leaving the right child nodes as sons. Thus in Figure \ref{fig:forests}, 2 is the left child of 3 in $T$ and it becomes the left brother of 3 in $\dec(T)$, 9 is a right child of 7 in $T$ and it stays the right-most child of 7 in $\dec(T)$. The increasing forest construction is then the ``right child = right brother'' bijection.

Also, the initial and final forests of a binary tree $T$ are indeed initial and final forests in the sense of interval-posets. In particular, they are interval-posets. The fact that they contain only increasing (resp. decreasing) relations is given by construction. It is left to check that they satisfy the Tamari axiom on all their elements: this is due to the binary search tree structure. In particular,  if you interpret a binary search tree as poset by pointing all edges toward the root then it is an interval-poset.  

\begin{Theorem}[from \cite{IntervalPosetsInitial} Thm 2.8]
Let $T_1$ and $T_2$ be two binary trees and $R = \rel(\dec(T_1)) \cup \rel(\inc(T_2))$. Then, $R$ is the set of relations of a poset $P$ if and only if $T_1 \leq T_2$ in the Tamari lattice. And in this case, $P$ is an interval-poset. 

This construction defines a bijection between interval-posets and intervals of the Tamari lattice.
\end{Theorem}

There are two ways in which $R$ could be not defining a poset. First, $R$ could be non-transitive. Because of the structure of initial and final forests, this never happens. Secondly, $R$ could be non-anti-symmetric by containing both $(a,b)$ and $(b,a)$ for some $a,b \leq n$. This happens if and only if $T_1 \not\leq T_2$. You can read more about this bijection in \cite{IntervalPosetsInitial}. Figure~\ref{fig:interval-poset-construction} gives an example.

\begin{figure}[ht]
\input{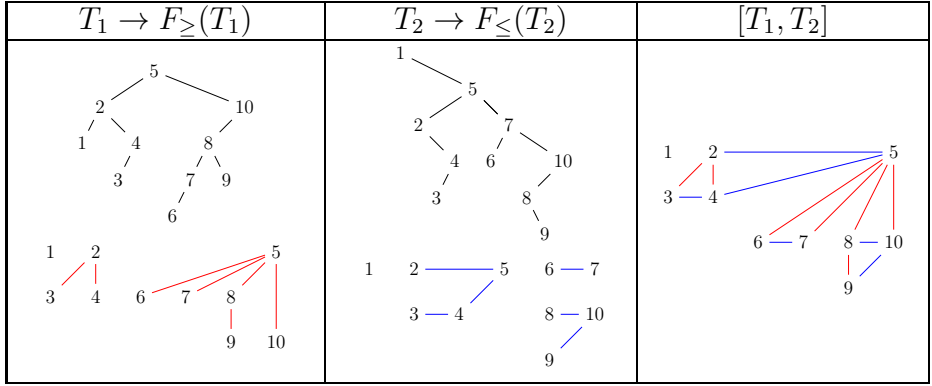}
\caption{Two trees $T_1 \leq T_2$ in the Tamari lattice and their corresponding interval-poset.}
\label{fig:interval-poset-construction}
\end{figure}

To better understand the relations between Tamari intervals and interval-posets, we now recall some results from \cite[Prop. 2.9]{IntervalPosetsInitial}, which are immediate from the construction of interval-posets and the properties of initial and final forests.

\begin{Proposition}[From \cite{IntervalPosetsInitial} Prop. 2.9]
\label{prop:interval-poset-extension}
Let $I$ and $I'$ be two interval-posets such that their respective Tamari intervals are given by $[A,B]$ and $[A',B']$, then
\begin{enumerate}
\item $I'$ is an extension of $I$ if and only if $A' \geq A$ and $B' \leq B$;
\label{prop-en:interval-poset-extension}
\item $I'$ is a decreasing-extension of $I$ if and only if $A' \geq A$ and $B' = B$;
\label{prop-en:interval-poset-decreasing-extension}
\item $I'$ is an increasing-extension of $I$ if and only if $A' = A$ and $B' \leq B$.
\label{prop-en:interval-poset-increasing-extension}
\end{enumerate}

\end{Proposition}

As the Tamari lattice is also often defined on Dyck paths, it is legitimate to wonder what is the direct bijection between a Tamari interval $[D_1, D_2]$ of Dyck paths and an interval-poset. Of course, one can just transform $D_1$ and $D_2$ into binary trees through the bijection of Definition~\ref{def:dyck-tree} and then construct the corresponding final and initial forests. But because many statistics we study in this paper are more naturally defined on Dyck paths than on binary trees, we give the direct construction.

Recall that for each up-step $d$ in a Dyck path, there is a corresponding down-step $d'$ which is the first step you meet by drawing a horizontal line starting from $d$. From this, one can define a notion of nesting: an up-step $d_2$ (and its corresponding down-step $d_2'$) is nested in $(d,d')$ if it appears in between $d$, $d'$ in the binary word of the Dyck path.
\begin{Proposition}
\label{prop:dyck-dec-forest}
Let $D$ be a Dyck path on which we apply the following process:
\begin{itemize}
\item label from 1 to $n$ all pairs of up-steps and their corresponding down-steps by reading the up-steps on the Dyck path from left to right,
\item define a poset $P$ by $b \trprec_P a$ if and only if $b$ is nested in $a$ in the previous labeling.
\end{itemize}
Then $\dec(D) := \dec(\tree(D)) = P$.
\end{Proposition}

This bijection is actually a very classical one. It consists of shrinking the Dyck path into a tree skeleton. In Figure~\ref{fig:dyck-dec-forest}, we show in parallel the process of Proposition~\ref{prop:dyck-dec-forest} on the Dyck path and the corresponding binary tree.
\begin{figure}[ht]
\input{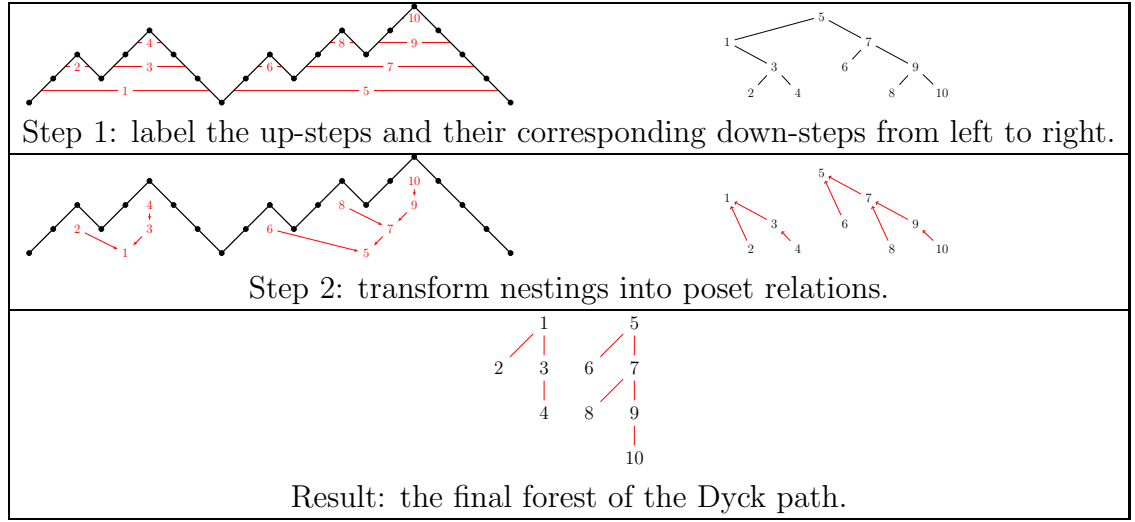}
\caption{Bijection between a Dyck path and its final forest.}
\label{fig:dyck-dec-forest}
\end{figure}

\begin{proof}
We use the recursive definition of the $\tree$ map. Let $D$ be a Dyck path. If $D$ is empty, then $\tree(D)$ is the empty binary tree and $\dec(D) = \dec(\tree(D))$ is the empty poset of size 0. If $D$ is a non-empty Dyck path, let $T = \tree(D)$. We want to check that $P$ is equal to $F := \dec(T)$. The path $D$ decomposes into $D = D_1 1 D_2 0$ with $\tree(D_1) = T_1$, the left subtree of $T$ and $\tree(D_2) = T_2$, the right subtree of $T$. We assume by induction that the proposition is true on $\dec(D_1)$ and $\dec(D_2)$. Let $1 \leq k \leq n$ be such that $\size(D_1) = k-1$ (in Figure~\ref{fig:dyck-dec-forest}, $k=5$): then $k$ is the label of the pair $(1,0)$ which appears in the decomposition of $D$. We also have that $v_k$ is the root of $T$. Now, let us choose $a < b \leq n$. Either
\begin{itemize}
\item $a < b < k$: the pairs of steps labeled by $a$ and $b$ both belong to $D_1$, we have $b \trprec_P a$ if and only if $b \trprec_F a$ by induction.
\item $b = k$: the pair labeled by $a$ belongs to $D_1$. It does not nest $k$, so $b \ntrprec_P a$. In $T$, $v_a$ is in $T_1$, the left subtree of $T$ and so we also have $b \ntrprec_F a$.
\item $ a < k < b$: then $a$ belongs to $D_1$ and $b$ belongs to $D_2$ In particular $b$ is not nested in $a$ and so $b \ntrprec_P a$. In $T$, $v_a$ is in $T_1$ and $v_b$ is in $T_2$. In particular, $v_b$ is not in the right subtree of $v_a$ and so $b \ntrprec_F a$. 
\item $a = k$: the pair labeled by $b$ belongs to $D_2$. It is nested in $k$, so $b \trprec_P a$. In $T$, $v_b$ belongs to $T_2$ the right subtree of $T$, we have $b \trprec_F a$.
\item $ k < a < b$: the pairs of steps labeled by $a$ and $b$ both belong to $D_2$, we have $b \trprec_P a$ if and only if $b \trprec_F a$ by induction.
\end{itemize}
\end{proof}

On binary trees, the constructions of the final and initial forests are completely symmetrical: the difference between the two only consists of a choice between left subtrees and right subtrees. Because the left-right symmetry of binary trees is not obvious when working on Dyck paths, the construction of the initial forest from a Dyck path gives a different algorithm than the final forest one.

\begin{Proposition}
\label{prop:dyck-inc-forest}
Let $D$ be a Dyck path of size $n$, we construct a directed graph following this process:
\begin{itemize}
\item label all up-steps of $D$ from 1 to $n$ from left to right,
\item for each up-step $a$, find, if any, the first up-step $b$ following the corresponding down-step of $a$ and add the edge $a \longrightarrow b$.
\end{itemize}
Then this resulting directed graph is the Hasse diagram of the initial forest of $D$.
\end{Proposition}

The construction is illustrated on Figure~\ref{fig:dyck-inc-forest}.

\begin{figure}[ht]
\input{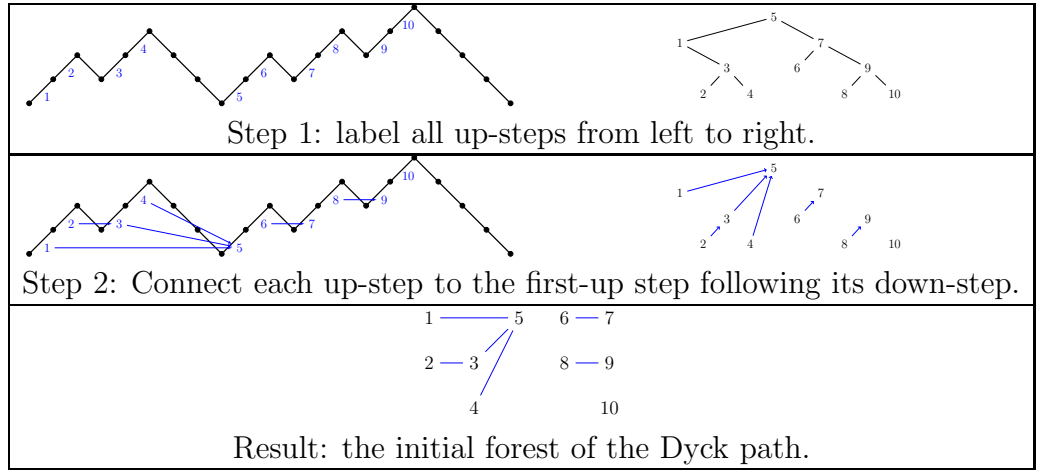}
\caption{Bijection between a Dyck path and its initial forest.}
\label{fig:dyck-inc-forest}
\end{figure}

\begin{proof}
We use the same induction technique as for the previous proof. The initial case is trivial. As before, when $D$ is non-empty, we have $D = D_1 1 D_2 0$ along with the corresponding trees $T$, $T_1$, and $T_2$ and $\size(D_1) = \size(T_1) = k-1$. We set $F := \inc(T)$ and we call $P$ the poset obtained by the algorithm.

First, let us prove that for all $a < k$, we have $a \trprec_P k$. Indeed suppose there exists $a < k$ with $a \ntrprec_P k$, we take $a$ to be maximal among those satisfying these conditions. We have $a \in D_1$ so its corresponding down-step appears before $k$, let $a' \leq k$ be the first up-step following the down-step of $a$. If $a' = k$, then $(a,k)$ is in the Hasse diagram of~$P$ and so $a \trprec_P k$. If $a' < k$, we have $a \trprec_P a'$ by definition and the maximality of $a$ gives $a' \trprec_P k$, which implies $a \trprec_P k$ by transitivity. 

Now let us choose $a < b \leq n$. Either
\begin{itemize}
\item $a < b < k$: the up-steps labeled by $a$ and $b$ both belong to $D_1$, we have $a \trprec_P b$ if and only if $a \trprec_F b$ by induction.
\item $b = k$: in $T$, $b$ is the root and $a$ is in its left subtree: we have $a \trprec_F b$. In $P$, we have also proved $a \trprec_P b$.
\item $a < k < b$: then $a$ belongs to $D_1$ and $b$ belongs to $D_2$ In particular $b$ is above $k$ in the path and there cannot be any link with $a$ even by transitivity, which means $a \ntrprec_P b$. In $T$, $v_a$ is in $T_1$ and $v_b$ is in $T_2$. In particular, $v_a$ is not in the left subtree of $v_b$ and so $a \ntrprec_F b$. 
\item $a = k$: the corresponding down-step of $a$ is the last step of $D$, which means there is no edge $(a,b)$ in $P$. Similarly, because $a$ is the tree root, there is no edge $(a,b)$ in $F$.
\item $ k < a < b$: the  up-steps labeled by $a$ and $b$ both belong to $D_2$, we have $a \trprec_P b$ if and only if $a \trprec_F b$ by induction.
\end{itemize}
\end{proof}

Now that we have described the relation between interval-posets and Tamari intervals both in terms of binary trees and Dyck path, we will often identify a Tamari interval with its interval-poset. When we refer to Tamari intervals in the future, we consider that they can be given indifferently by a interval-poset or by a couple of a lower bound and an upper bound $[A,B]$ where $A$ and $B$ can either be binary trees or Dyck paths.

\section{Statistics}
\label{sec:statistics}

\subsection{Statement of the main result}
\label{sec:statement}

\begin{Definition}
\label{def:contact-rise-dw}
Let $D$ be a Dyck path. 
\begin{itemize}
\item $\contacts(D)$ is the number of non-final contacts of the path $D$: the number of time the path $D$ touches
the line $y=0$ outside the final point.

\item $\rises(D)$ is the initial rise of $D$: the number of initial consecutive up-steps.

\item Let $u_i$ be the $i^{th}$ up-step of $D$, we consider the maximal subpath starting right after $u_i$ which is a Dyck path. Then the  \emph{contacts of $u_i$}, $\contactsStep{i}(D)$, are the number of non-final contacts of this Dyck path .

\item Let $v_i$ be the $i^{th}$ down-step of $D$, we say that the number of consecutive up-steps right after $v_i$ are the \emph{rises} of $v_i$ and write~$\risesStep{i}(D)$.

\item $\contactsV(D) := (\contacts(D), \contactsStep{1}(D), \dots, \contactsStep{n-1}(D))$ is the \emph{contact vector} of~$D$.

\item $\contactsV^*(D) := (\contactsStep{1}(D), \dots, \contactsStep{n-1}(D))$ is the \emph{truncated contact vector} of~$D$.

\item $\risesV(D) := (\rises(D), \risesStep{1}(D), \dots, \risesStep{n-1}(D))$ is the \emph{rise vector} of~$D$.

\item $\risesV^*(D) := (\risesStep{1}(D), \dots, \risesStep{n-1}(D))$ is the \emph{truncated rise vector} of~$D$.

\item Let  $X = (x_0, x_1, x_2, \dots)$ be a commutative alphabet, we write $\contactsP(D,X)$ the monomial $x_{\contacts(D)}, x_{\contactsStep{1}(D)}, \dots, x_{\contactsStep{n-1}(D)}$ and we call it the \emph{contact monomial} of $D$.

\item Let  $Y = (y_0, y_1, y_2, \dots)$ be a commutative alphabet, we write $\risesP(D,Y)$ the monomial $y_{\rises(D)}, y_{\risesStep{1}(D)}, \dots, y_{\risesStep{n-1}(D)}$ and we call it the \emph{rise monomial} of $D$.
\end{itemize}
\end{Definition}

\begin{figure}[ht]
\input{figures/contact_example}
\caption{Contacts and rises of a Dyck path}
\label{fig:contact-rise-example}
\end{figure}

Figure~\ref{fig:contact-rise-example} gives an example of the different contacts and rises values computed on a given Dyck path. The Dyck path can be easily reconstructed from $\risesV(D)$. This is also true of $\contactsV(D)$ even though it is less obvious. It will become clear once we express the statistics in terms of planar forests. At first, let us use the definitions on Dyck paths to express our main result on Tamari intervals.

\begin{Definition}
\label{def:contact-rise-intervals}
Consider an interval $I$ of the Tamari lattice described by two Dyck paths $D_1$ and $D_2$ with $D_1 \leq D_2$. Then
\begin{enumerate}
\item $\contactsStep{i}(I):= \contactsStep{i}(D_1)$ for $0 \leq i \leq n$, $\contactsV(I):=\contactsV(D_1)$, $\contactsV^*(I):=\contactsV^*(D_1)$, and $\contactsP(I,X):=\contactsP(D_1,X)$;

\item $\risesStep{i}(I):=\risesStep{i}(D_2)$ for $0 \leq i \leq n$, $\risesV(I):=\risesV(D_2)$, $\risesV^*(I):=\risesV^*(D_2)$ and $\risesP(I,Y) := \risesP(D_2,Y)$.
\end{enumerate}
To summarize, all the statistics we defined on Dyck paths are extended to Tamari intervals by looking at the \emph{lower bound} Dyck path $D_1$ when considering contacts and the \emph{upper bound} Dyck path $D_2$ when considering rises. 
\end{Definition}

Most of these statistics have been considered before on both Dyck paths and Tamari intervals. In \cite{mTamari}, one can find the same definitions for the initial rise $\rises(I)$ and number of non-final contacts $\contacts(I)$. Taking $x_0 = y_0 = 1$ in $\contactsP(I,X)$ and $\risesP(I,Y)$ corresponds to ignoring $0$ values in $\contactsV(I)$ and $\risesV(I)$: we find those monomials in Préville-Ratelle's thesis \cite{PRThesis}. Our definition of $\contactsP(I,X)$ is slightly different than the one of Préville-Ratelle: we will explain the correspondence in the more general case of $m$-Tamari intervals in Section~\ref{sec:mtam}. We now describe another statistic from \cite{PRThesis} which is specific to Tamari intervals: it cannot be defined through a Dyck path statistics on the interval end points. 

\begin{Definition}
\label{def:distance}
Let $I = [D_1, D_2]$ be an interval of the Tamari lattice. A \emph{chain} between $D_1$ and $D_2$ is a list of Dyck paths
\begin{equation*}
D_1 = P_1 < P_2 < \dots < P_k = D_2
\end{equation*}
which connects $D_1$ and $D_2$ in the Tamari lattice. If the chain comprises $k$ elements, we say it is of length $k-1$ (the number of cover relations).

We call the \emph{distance} of $I$ and write $\distance(I)$ the maximal length of all chains between $D_1$ and $D_2$. 
\end{Definition}

For example, if $I = [D,D]$ is reduced to a single element, then $\distance(I) = 0$. If $I = [D_1, D_2]$ and $D_1 \leq D_2$ is a cover relation of the Tamari lattice, then $\distance(I) = 1$. This statistic was first described in \cite{BergmTamari}. It generalizes the notion of \emph{area} of a Dyck path to an interval. To finish, we need the notation $\size(I)$, which is defined to be the size of the elements of $I$: if $I$ is an interval of Dyck paths of size $n$, then $\size(I) = n$. Note that it is also the number of vertices of the interval-poset representing $I$. We can now state the first version of the main result of this paper.

\begin{Theorem}[classical case]
\label{thm:main-result-classical}
Let $x,y,t,q$ be variables and $X = (x_0, x_1, x_2, \dots)$ and $Y = (y_0, y_1, y_2, \dots)$ be commutative alphabets. Consider the generating function

\begin{equation}
\Phi(t; x, y, X, Y, q) = \sum_{I} t^{\size(I)} x^{\contacts(I)} y^{\rises(I)} \contactsP(I,X) \risesP(I,Y) q^{\distance(I)}
\end{equation}  

summed over all intervals of the Tamari lattice. Then we have

\begin{equation}
\Phi(t; x, y, X, Y, q) = \Phi(t; y, x, Y, X, q).
\end{equation}
\end{Theorem}

For $x_0 = y_0 = 1$, this corresponds to a special case of \cite[Conjecture 17]{PRThesis} where $m=1$, the general case will be dealt in Section~\ref{sec:mtam}. The case where $X,Y,$ and $q$ are set to 1 is proved algebraically in \cite{mTamari}. In this paper, we give a combinatorial proof by describing an involution on Tamari intervals that switches $\contacts$ and $\rises$ as well as $\contactsP$ and $\risesP$. The involution is described in Section~\ref{sec:involutions}. 

One corollary of Theorem~\ref{thm:main-result-classical} is that the symmetry also exists when we restrict the sum to Dyck paths,

\begin{equation}
 \sum_{D} P_D(t,X,Y)  = \sum_{D} P_D(t,Y,X),
\end{equation}
where $P_D(t,X,Y) = t^{\size(D)} x^{\contacts(D)} y^{\rises(D)} \contactsP(D,X) \risesP(D,Y)$, summed over all Dyck paths. Indeed, an interval with distance 0 is reduced to a single element and, in this case, the statistics of the interval correspond to the classical statistics on the Dyck path. This particular case can be proved directly by conjugating two very classical involutions on Dyck path: the reversing of the Dyck path and the Tamari symmetry. We illustrate this in Figure~\ref{fig:inv_dyckpaths}. What we call the ``Tamari symmetry'' is the natural involution that is given by the top-down symmetry of the Tamari lattice itself. It is described more directly on binary trees, where it corresponds to recursively switching left and right subtrees. The Tamari symmetry is by nature compatible with the Tamari order and can be directly generalized to intervals. This is not the case of the reversal of Dyck path. In other words, if two Dyck paths are such that $D_1 \leq D_2$ in the Tamari lattice, then in general $D_1'$ is not comparable to $D_2'$, where $D_1'$ and $D_2'$ are the reverse Dyck paths of $D_1$ and $D_2$ respectively. This is exactly where lies the difficulty in finding the rise-contact involution on Tamari intervals: the transformation of $D_1$ and $D_2$ are inter correlated. Basically, we have found a way to reverse $D_2$ by keeping track of $D_1$. First, let us interpret the statistics directly on interval-posets.

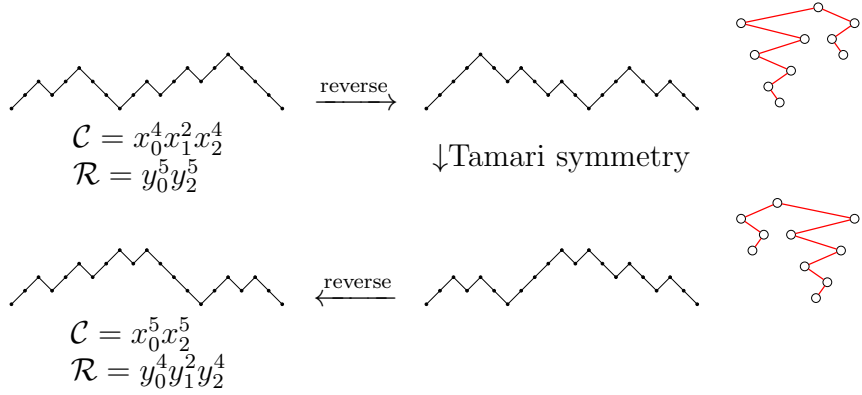
\begin{figure}[ht]
\input{figures/involution_dyckpaths}
\caption{The rise-contact involution on Dyck paths}
\label{fig:inv_dyckpaths}
\end{figure}

\begin{Definition}
Let I be an interval-poset of size $n$, we define
\begin{itemize}
\item $\dcstep{0}(I)$ (resp. $\icinf(I)$) is the number of decreasing (resp. increasing) roots of $I$.
\item $\dcstep{i}(I)$ (resp. $\icstep{i}(I)$) for $1 \leq i \leq n$ is the number of decreasing (resp. increasing) children of the vertex $i$.
\item $\DC(I) := (\dcstep{0}(I), \dcstep{1}(I), \dots, \dcstep{n-1}(I))$ is called the \emph{final forest vector} of $I$ and $\DC^*(I) := (\dcstep{1}(I), \dots, \dcstep{n-1}(I))$ is the \emph{truncated final forest vector}.
\item $\IC(I) := (\icinf(I), \icstep{n}(I), \dots, \icstep{2}(I))$ is called the \emph{initial forest vector} of $I$ and $\IC^*(I) := (\icstep{n}(I), \dots, \icstep{2}(I))$ is the \emph{truncated initial forest vector}.
\end{itemize}
\end{Definition}

Note that we do not include $\dcstep{n}$ nor $\icstep{1}$ in the corresponding vectors as they are always 0. The vertices of $I$ are read in their natural order in $\DC$ and in reverse order in $\IC$: this follows a natural traversal of the final (resp. initial) forests from roots to leaves. As an example, in Figure~\ref{fig:interval-poset-example2}, we have $\DC(I) = \left(3, 0,2,0,0,4, 0,0,1, 0 \right)$ and $\IC(I) = \left( 4, 2,0,0,1,0,2,1,0,0 \right)$.

\begin{Proposition}
\label{prop:contact-dc}
Let $I$ be an interval-poset, then $\DC(I) = \contactsV(I)$.
\end{Proposition}
\begin{proof}
This is clear from the construction of the final forest from the Dyck path given in Proposition~\ref{prop:dyck-dec-forest}. Indeed, each non-final contact of the Dyck path corresponds to exactly one decreasing root of the interval-poset. Then the decreasing children of a vertex are the contacts of the Dyck path nested in the corresponding (up-step, down-step) tuple.
\end{proof}

\begin{Remark}
The vector $\IC(I)$ is not equal to $\risesV(I)$ in general. In fact, the interpretation of rises directly on the interval-poset is not easy. What we will prove anyway is that the two vectors can be exchanged through an involution on $I$. This involution is shown in Section~\ref{sec:involutions} and is a crucial step in proving Theorem~\ref{thm:main-result-classical}.
\end{Remark}

\subsection{Distance and Tamari inversions}
\label{sec:tamari-inversions}

Before describing the involutions used to prove Theorem~\ref{thm:main-result-classical}, we discuss more the \emph{distance} statistics on Tamari intervals in order to give a direct interpretation of it on interval-posets.

\begin{Definition}
\label{def:tamari-inversions}
Let $I$ be an interval-poset of size $n$. A pair $(a,b)$ with $ 1 \leq a < b \leq n$ is said to be a \emph{Tamari inversion} of $I$ when
\begin{itemize}
\item there is no $a \leq k < b$ with $b \trprec k$;
\item there is no $a < k \leq b$ with $a \trprec k$. 
\end{itemize} 
We write $\TInv(I)$ the set of Tamari inversions of a set $I$.
\end{Definition}

As an example, the Tamari inversions of the interval-poset of Figure~\ref{fig:interval-poset-example2} are exactly $(1,2), (1,5), (7,8), (7,10)$. As counter-examples, you can see that $(1,6)$ is not a Tamari inversion because we have $1 < 5 < 6$ and $6 \trprec 5$. Similarly, $(6,8)$ is not a Tamari inversion because there is $6 < 7 < 8$ and $6 \trprec 7$. Note also that if $(a,b)$ is a Tamari inversion of~$I$, then $a \ntrprec b$ and $b \ntrprec a$. Our goal is to prove the following statement.

\begin{Proposition}
\label{prop:tamari-inversions}
Let $I$ be an interval-poset, then $\distance(I)$ is equal to the number of Tamari inversions of $I$.
\end{Proposition}

The proof of Proposition~\ref{prop:tamari-inversions} requires two inner results that we express as Lemmas.

\begin{Lemma}
\label{lem:tamari-inversions-extension}
Let $I$ be an interval-poset whose Tamari interval is given by $[T_1, T_2]$ where $T_1$ and $T_2$ are binary trees. Let $I'$ be another interval given by $[T_1',T_2]$ with $T_1' > T_1$ in the Tamari lattice. Then the interval-poset of $I'$ is an extension of $I$ such that if we have $a < b$ with $(b,a)$ a decreasing-cover relation of $I'$ with $b \ntrprec_I a$, then $(a,b)$ is a Tamari inversion of $I$. In other words, $I'$ can be obtained from $I$ by adding only decreasing relations given by some Tamari inversions.
\end{Lemma}

\begin{proof}
By Proposition~\ref{prop:interval-poset-extension}, we know that $I'$ is a decreasing-extension of $I$. This Lemma is then just a refinement of Proposition~\ref{prop:interval-poset-extension}, which states that the decreasing relations that have been added come from the Tamari inversions of $I$.

Let $(b,a)$ be a decreasing-cover relation of $I'$ such that $b \ntrprec_I a$. Because $I'$ is an extension of $I$, we also know that $a \ntrprec_I b$. Let $k$ be such that $a < k < b$. Because we have $b \trprec_{I'} a$, the Tamari axiom on $a,k,b$ gives us $k \trprec_{I'} a$. This implies that $b \ntrprec_{I'} k$ as $(b,a)$ is a decreasing-cover relation of $I'$ by hypothesis. In particular, we cannot have $b \trprec_I k$ either as any relation of $I$ is also a relation of $I'$. Similarly, we cannot have $a \trprec_I k$ as this would imply $a \trprec_{I'} k$, contradicting $k \trprec_{I'} a$.
\end{proof}

\begin{Lemma}
\label{lem:tamari-inversions-adding}
Let $I$ be an interval-poset such that $\TInv(I) \neq \emptyset$ and let $(a,b)$ be its first Tamari inversion in lexicographic order. Then by adding the relation $(b,a)$ to $I$, we obtain an interval-poset $I'$ such that the number of Tamari inversions of $I'$ is the number of Tamari inversions of $I$ minus one. 
\end{Lemma}

\begin{proof}
Because $(a,b)$ is a Tamari inversion of $I$, we have $b \ntrprec_I a$ and $a \ntrprec_I b$, which means the relation $(b,a)$ can be added to~$I$ as a poset. We need to check that the result $I'$ is still an interval-poset.

Let us first prove that for all $k$ such that $a < k < b$, we have $k \trprec_I a$. Let us suppose by contradiction that there exist $a < k < b$ with $k \ntrprec_I a$ and let us take the minimal $k$ possible. Note that $(a,k)$ is smaller than $(a,b)$ in the lexicographic order, which implies that $(a,k)$ is not a Tamari inversion. If there is $k'$ such that $a < k' \leq k$ with $a \trprec_I k'$ then $(a,b)$ is not a Tamari inversion. So there is $k'$ with $a \leq k' < k$ with $k \trprec_I k'$. But because we took $k$ minimal, we get $k' \trpreceq_I a$, which implies $k \trprec_I a$ and contradicts the fact that $(a,b)$ is a Tamari inversion. 

Now, we show that the Tamari axiom is satisfied for all $a'$, all $k$, and all $b$ such that $a' < k < b'$. By Remark~\ref{rem:adding-decreasing}, we only have to consider decreasing relations. More precisely, the only cases to check are the ones where $b' \ntrprec_I a'$ and $b' \trprec_{I'} a'$, which means $a \trpreceq_I a'$ and $b' \trpreceq_I b$ (the relation is either directly added through $(b,a)$ or obtained by transitivity). Let us choose such a couple $(a',b')$ and first prove that $a' \leq a < b \leq b'$.

\begin{itemize}
\item $b' \neq a$ and $a' \neq b$ because both would imply $a \trprec_I b$, which contradicts the fact that $(a,b)$ is a Tamari inversion.

\item If $b' < a$, we have $b' < a < b$ and $b' \trprec_I b$, which implies $a \trprec_I b$ by the Tamari axiom on $(b',a,b)$. This contradicts the fact that $(a,b)$ is a Tamari inversion. 

\item If $a < b' < b$, we have proved $b' \trprec_I a$, which gives $b' \trprec_I a'$ by transitivity and contradicts our initial hypothesis.

\item If $b < a'$, we have $a < b < a'$ with $a \trprec_I a'$, which implies $b \trprec_I a'$ by the Tamari axiom on $(a,b,a')$. This gives $b' \trprec_I a'$ by transitivity and contradicts our initial hypothesis.

\item If $a < a' < b$ then we have $a \trprec a'$ and $(a,b)$ is not a Tamari inversion.
\end{itemize}

We now have $a' \leq a < b \leq b'$. Now for $k$ such that $a' < k < b'$, if $k < a$ we get $k \trprec_I a'$ by the Tamari axiom on $(a',k,a)$. If $a < k < b$, we have proved that $k \trprec_I a$ and so $k \trprec_I a'$ by transitivity. If $b < k < b'$, the Tamari axiom on $(b,k,b')$ gives us $k \trprec_I b$, which gives in $I'$ $k \trprec_{I'} b \trprec_{I'} a \trprec_{I'} a'$ so $k \trprec_{I'} a'$ by transitivity. In all cases, the Tamari axiom is satisfied in $I'$ for $(a',k,b')$.

There is left to prove that the number of Tamari inversions of $I'$ has been reduced by exactly one. More precisely: all Tamari inversions of~$I$ are still Tamari inversions of $I'$ except $(a,b)$. Let $(a',b')$ be another Tamari inversion of $I$. Because $(a,b)$ is minimal in lexicographic order, we have either $a' > a$ or $a' = a$ and $b' > b$.
\begin{itemize}
\item If $a' > a$, let $k$ be such that $a' \leq k < b'$. We have $b' \ntrprec_I k$. Suppose that we have $b' \trprec_{I'} k$, which means that it has been added by transitivity and so we have $b' \trprec_{I} b$ and $a \trprec_I k$. Because $(a,b)$ is a Tamari inversion of $I$, we get that $k > b$. We have $a < b < k$ and $b < k < b'$, the Tamari axioms on $(a,b,k)$ and $(b',k,b)$ leads to a contradiction in $I$. Now, let $k$ be such that $a' < k \leq b'$. We have $a' \ntrprec_I k$. No increasing relation has been created in $I'$ and so $a' \ntrprec_{I'} k$.
\item If $a = a'$ and $b' > b$, first note that $b' \ntrprec_I b$. Indeed we have $a' < b < b'$ and this would contradict the fact that $(a',b')$ is a Tamari inversion. Let $k$ be such that $a \leq k < b'$, then $b' \ntrprec_{I} k$. Because $b' \ntrprec_{I} b$, the relation $(b',k)$ cannot be obtained by transitivity in $I'$ and so $b' \ntrprec_{I'} k$. Now, if $a < k \leq b'$, we have $a \ntrprec_{I} k$ and by the same argument as earlier that no increasing relation has been created in $I'$, $a \ntrprec_{I'} k$.
\end{itemize}
\end{proof}

\begin{proof}[Proof of Proposition~\ref{prop:tamari-inversions}]
Let $I$ be an interval-poset containing $v$ Tamari inversions and whose bounds are given by two binary trees $[T_1, T_2]$. Suppose there is a chain of length $k$ between $T_1$ and $T_2$. In other words, we have $k+1$ binary trees
\begin{align*}
T_1 = P_1 < P_2 < \dots < P_{k+1} = T_2
\end{align*}
which connects $T_1$ and $T_2$ in the Tamari lattice. Let us look at the intervals $[P_i, T_2]$. Lemma~\ref{lem:tamari-inversions-extension} tells us that each of them can be obtained by adding decreasing relations $(b,a)$ to $I$ where $(a,b) \in \TInv(I)$. We now apply Proposition~\ref{prop:interval-poset-extension}. In our situation, it means that, for $1 \leq j \leq k+1$, the interval-poset of $[P_j,T_2]$ is an extension of every interval-posets $[P_i, T_2]$ with $1 \leq i \leq j$: the Tamari inversions that were added as decreasing relations in $[P_i,T_2]$ are kept in $[P_j,T_2]$. In other words, to obtain $P_{i+1}$ from $P_{i}$, one or more Tamari inversions of $I$ are added to $P_i$ as decreasing relations. At least one Tamari inversion is added at each step, which implies that $v \geq k$. This is true for all chain and thus $v \geq \distance(I)$.

Now, let us explicitly construct a chain between $T_1$ and $T_2$ of length~$v$. This will give us that $v \leq \distance(I)$ and conclude the proof. We proceed inductively.
\begin{itemize}
\item If $v = 0$, then $\distance(I) \leq v$ is also 0, which means $T_1 = T_2$: this is a chain of size 0 between $T_1$ and $T_2$.
\item We suppose $v > 0$ and we apply Lemma~\ref{lem:tamari-inversions-adding}. We take the first Tamari inversion of $\TInv(I)$ in lexicographic order and add it to~$I$ as a decreasing relation. We obtain an interval-poset $I'$ which is a decreasing-extension of $I$ with $v-1$ Tamari inversions. Then by Proposition~\ref{prop:interval-poset-extension}, the bounds of $I'$ are given by $[T_1',T_2]$ with $T_1' > T_1$. By induction, we construct a chain of size $v-1$ between $T_1'$ and $T_2$, which gives us a chain of size $v$ between $T_1$ and $T_2$.
\end{itemize}
\end{proof}

The interpretation of the distance of an interval as a direct statistic on interval-posets is very useful for our purpose here as it gives an explicit way to compute it and its behavior through our involutions will be easy to state and prove. It is also interesting in itself. Indeed, this statistic appears in other conjectures on Tamari intervals, for example Conjecture~19 of \cite{PRThesis}, which is related to the well known open $q$-$t$-Catalan problems.

\section{Involutions}
\label{sec:involutions}

\subsection{Grafting of interval-posets}
\label{sec:composition}

In this section, we revisit some major results of \cite{IntervalPosetsInitial} which we will be used to define some new involutions. 

\begin{Definition}
Let $I_1$ and $I_2$ be two interval-posets, we define a \emph{left grafting} operation and a \emph{right grafting} operation depending on a parameter $r$. Let $\alpha$ and $\omega$ be respectively the label of minimal value of $I_2$ (shifted by the size of $I_1$) and the label of maximal value of $I_1$. Let $c = \contacts(I_2)$ and $y_1, \dots y_c$ be the decreasing roots of $I_2$ (shifted by the size of $I_1$).

The left grafting of $I_1$ over $I_2$ with $\size(I_2) > 0$ is written as $I_1 \pleft I_2$. It is defined by the shifted concatenation of $I_1$ and $I_2$ along with relations $y \trprec \alpha$ for all $y \in I_1$. 

The right grafting of $I_2$ over $I_1$ with $\size(I_1) > 0$ is written as $I_1 \pright{r} I_2$ with $0 \leq r \leq c$. It is defined by the shifted concatenation of $I_1$ and $I_2$ along with relations $y_i \trprec \omega$ for $1 \leq i \leq r$. 
\end{Definition}

\begin{figure}[ht]
\input{figures/grafting}
\caption{Grafting of interval-posets}
\label{fig:grafting}
\end{figure}
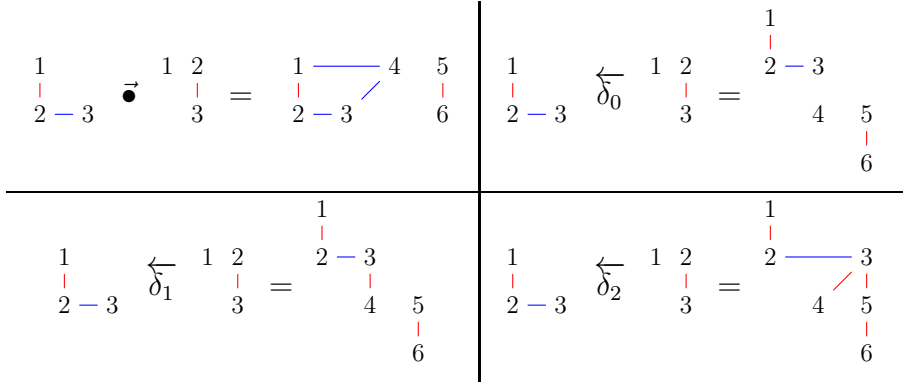

Figure~\ref{fig:grafting} gives an example. Note that the vertices of $I_2$ are always shifted by the size of $I_1$. For simplicity, we do not always recall this shifting: when we mention a vertex $x$ of $I_2$ in a grafting, we mean the shifted version of $x$. These two operations were defined in \cite[Def. 3.5]{IntervalPosetsInitial}. Originally, the right grafting was defined as a single operation $\pright{}$ whose result was a formal sum of interval-posets. In this paper, it is more convenient to cut it into different sub-operations depending on a parameter. We can use these operations to uniquely \emph{decompose} interval-posets: this will be explained in Section~\ref{sec:decomposition}. First, we will study how the different statistics we have defined are affected by the operations. We start with the contact vector $\contactsV$, which is equal to the final forest vector $\DC$.

\begin{Proposition}
\label{prop:grafting-contact}
Let $I_1$ and $I_2$ be two interval-posets of respective sizes $n > 0$ and $m > 0$, then
\begin{itemize}
\item $\contacts(I_1 \pleft I_2) = \contacts(I_1) + \contacts(I_2)$;
\item $\contactsV(I_1 \pleft I_2) = \left( \contacts(I_1) + \contacts(I_2), \contactsStep{1}(I_1), \dots, \contactsStep{n-1}(I_1),0, \contactsStep{1}(I_2), \dots, \contactsStep{m-1}(I_2) \right)$;
\item $\contacts(I_1 \pright{i} I_2) = \contacts(I_1) + \contacts(I_2) - i$;
\item $\contactsV(I_1 \pright{i} I_2) = \left( \contacts(I_1) + \contacts(I_2) - i, \contactsStep{1}(I_1), \dots, \contactsStep{n-1}(I_1),i, \contactsStep{1}(I_2), \dots, \contactsStep{m-1}(I_2) \right)$.
\end{itemize}

If $\size(I_1)= 0$ then $I_1 \pleft I_2 = I_2$ and $\contactsV(I_1 \pleft I_2) = \contactsV(I_2)$. If $\size(I_2)= 0$ then $I_1 \pright{i} I_2 = I_1$ and $\contactsV(I_1 \pright{i} I_2) = \contactsV(I_1)$.
\end{Proposition}

This can be checked on Figure~\ref{fig:grafting}. We have $\contactsV(I_1) = \DC(I_1) = (2,1,0)$ because there are two connected components in the final forest ($2 \trprec 1$ and $3$) and $1$ and $2$ have respectively 1 and 0 decreasing children. For $I_2$, we get $\contactsV(I_2) = (2,0,1)$. Now, it can be checked that $\contactsV(I_1 \pleft I_2) = (4,1,0,0,0,1)$, $\contactsV(I_1 \pright{0} I_2) = (4,1,0,0,0,1)$, $\contactsV(I_1 \pright{1} I_2) = (3,1,0,1,0,1)$, $\contactsV(I_1 \pright{2} I_2) = (2,1,0,2,0,1)$. 

\begin{proof}
First, remember that, by Proposition~\ref{prop:contact-dc}, contacts can be directly computed on the final forest of the interval-posets: the non-final contacts correspond to the number of components and $\contactsStep{v}$ for $1 \leq v \leq n$ is the number of decreasing children of the vertex $v$. 

Now, in the left grafting $I_1 \pleft I_2$, the two final forests are simply concatenated. In particular, $\contacts(I_1 \pleft I_2) = \contacts(I_1) + \contacts(I_2)$. The contact vector $\contactsV(I_1 \pleft I_2)$ is then formed by this initial value followed by the truncated contact vector of $I_1$, then an extra 0, which corresponds to $\contactsStep{n}$, then the truncated contact vector of $I_2$.

The contacts of the right grafting $I_1 \pright{i} I_2$ depend on the parameter~$i$. Indeed, each added decreasing relation merges one component of the final forest of $I_2$ with the last component of the final forest of $I_1$ and thus reduces the number of components by one. As a consequence, we have $\contacts(I_1 \pright{i} I_2) = \contacts(I_1) + \contacts(I_2) - i$. The contact vector is formed by this initial value followed by the truncated contact vector of $I_1$, then the new number of decreasing children of $n$, which is $i$ by definition, then the truncated contact vector of $I_2$.
\end{proof}

Let us now study what happens to the rise vector $\risesV$ and the initial forest vector $\IC$. They both only depend of the initial forest (increasing relations). The vector $\IC$ can be read directly on the interval-poset and we get the following proposition.

\begin{Proposition}
\label{prop:grafting-ic}
Let $I_1$ and $I_2$ be two interval-posets of respective sizes $n >0$ and $m >0$, then
\begin{itemize}
\item $\icinf(I_1 \pleft I_2) = \icinf(I_2)$;
\item $\IC(I_1 \pleft I_2) = \left( \icinf(I_2), \icstep{m}(I_2), \dots, \icstep{2}(I_2), \icinf(I_1), \icstep{n}(I_1), \dots \icstep{2}(I_1) \right)$;
\item $\icinf(I_1 \pright{i} I_2) = \icinf(I_2) + \icinf(I_1)$;
\item $\IC(I_1 \pright{i} I_2) = \left(\icinf(I_2) + \icinf(I_1), \icstep{m}(I_2), \dots, \icstep{2}(I_2), 0, \icstep{n}(I_1), \dots, \icstep{2}(I_1) \right)$.
\end{itemize}

If $\size(I_1)= 0$ then $I_1 \pleft I_2 = I_2$ and $\IC(I_1 \pleft I_2) = \IC(I_2)$. If $\size(I_2)= 0$ then $I_1 \pright{i} I_2 = I_1$ and $\IC(I_1 \pright{i} I_2) = \IC(I_1)$.
\end{Proposition}

This can be checked on Figure~\ref{fig:grafting}. We initially have $\IC(I_1) = (2,1,0)$ and $\IC(I_2) = (3,0,0)$, and then $\IC(I_1 \pleft I_2) = (3,0,0,2,1,0)$ and $\IC(I_1 \pright{i} I_2) = (5, 0, 0, 0, 1, 0)$ for all $1 \leq i \leq 2$. 

\begin{proof}
When we compute $I_1 \pleft I_2$, we add increasing relations from all vertices of $I_1$ to the first vertex $\alpha$ of the shifted copy of $I_2$. In other words, we attach all increasing roots of $I_1$ to a new root $\alpha$. The number of components in the initial forest of $I_1 \pleft I_2$ is then given by $\icinf(I_2)$ (the last component contains $I_1$) and the number of increasing children of $\alpha$ is given by $\icinf(I_1)$. Other number of increasing children are left unchanged and we thus obtain the expected vector.

In the computation of $I_1 \pright{i} I_2$, the value of $i$ only impacts the decreasing relations and thus does not affects the vector $\IC$. No increasing relation is added, which means that the initial forests of $I_1$ and $I_2$ are only concatenated and by looking at connected components, we obtain $\icinf(I_1 \pright{i} I_2) = \icinf(I_1) + \icinf(I_2)$. The vector $\IC$ is formed by this initial value followed by the truncated initial forest vector of $I_2$, then an extra 0, which correspond to $\icstep{1}(I_2)$, then the truncated initial forest vector of $I_1$. 
\end{proof}

To understand how the rise vector behaves through the grafting operations, we first need to interpret the grafting on the upper bound Dyck path of the interval. We start with the left grafting.

\begin{Proposition}
\label{prop:grafting-rise-left}
Let $I_1$ and $I_2$ be two interval-posets of respective sizes $n > 0$ and $m >0$. Let $D_1$ and $D_2$ be their respective upper bound Dyck path. Then, the upper-bound Dyck path of $I_1 \pleft I_2$ is given by $D_1D_2$ and consequently, if $\size(I_1) > 0$, we get
\begin{itemize}
\item $\rises(I_1 \pleft I_2) = \rises(I_1)$;
\item $\risesV(I_1 \pleft I_2) = (\rises(I_1), \risesStep{1}(I_1), \dots, \risesStep{n-1}(I_1), \rises(I_2), \risesStep{1}(I_2), \dots, \risesStep{m-1}(I_2))$.
\end{itemize}
\end{Proposition} 

\begin{figure}[ht]
\input{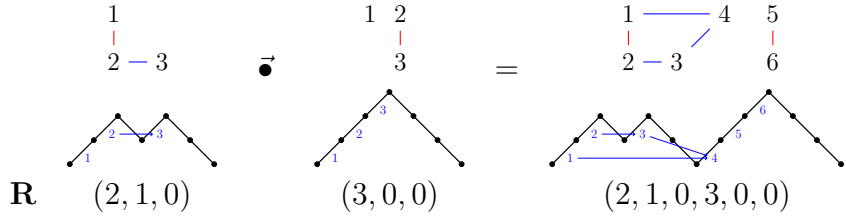}
\caption{The upper bound Dyck paths in the left grafting}
\label{fig:grafting-rise-left}
\end{figure}

Figure~\ref{fig:grafting-rise-left} gives an example of left-grafting with corresponding upper bound Dyck paths and rise vectors.

\begin{proof}
The definition of $I_1 \pleft I_2$ states that we add all relations $(i,\alpha)$ with $i \in I_1$ and $\alpha$ the first vertex of $I_2$. This is the same as adding all relations $(i, \alpha)$ where $i$ is an increasing root of $I_1$ (the other relations are obtained by transitivity). The increasing roots of $I_1$ correspond to the up-steps of $D_1$ whose corresponding down-steps do not have a following up-step, \emph{i.e.}, the up-steps corresponding to final down-steps of $D_1$. By concatenating $D_1$ and $D_2$, the first up-step of $D_2$ is now the first following up-step of the final down-steps of $D_1$: this indeed adds the relations from the increasing roots of $I_1$ to the first vertex of $I_2$. The expressions for the initial rise and rise vectors follow immediately by definition.
\end{proof}

The effect of the right grafting on the rise vector is a bit more technical. For simplicity, we only study the case where $I_1$ is of size one, which is the only case we will need in this paper.

\begin{Proposition}
\label{prop:grafting-rise-right}
Let $I$ be an interval-poset of size $n >0$ and $D$ its upper bound Dyck path. Let $u$ be the only interval-poset on a single vertex. Note that the upper bound Dyck path of $u$ is given by the word $10$. Then, the upper bound Dyck path of $u \pright{i} I$ is $1D0$ for all $0 \leq i \leq \contacts(I)$, and we have
\begin{itemize}
\item $\rises(u \pright{i} I) = \rises(I) + 1$;
\item $\risesV(u \pright{i} I) = (\rises(I) + 1, \risesStep{1}(I), \dots, \risesStep{n-1}(I), \risesStep{n}(I) = 0)$.
\end{itemize} 

\end{Proposition}

\begin{figure}[ht]
\input{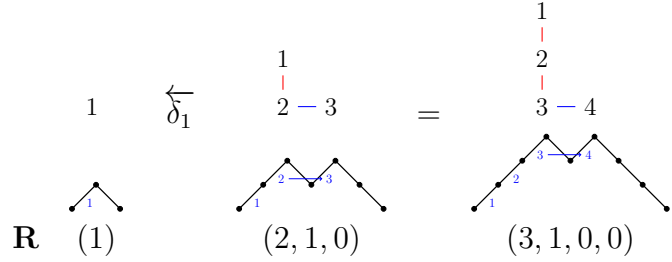}
\caption{The upper bound Dyck paths in the right grafting}
\label{fig:grafting-rise-right}
\end{figure}

Figure~\ref{fig:grafting-rise-right} gives an example of right-grafting with corresponding upper bound Dyck paths and rise vectors.

\begin{proof}
The right-grafting only adds decreasing relations. On the initial forests, it is then nothing but a concatenation of the two initial forests. In particular, in the case of $u \pright{i} I$, no increasing relation is added from the vertex one to any vertex of $I$. On the upper bound Dyck path, this means that the down-step corresponding to the initial up-step is not followed by any up-step: the Dyck path of $I$ has to be nested into this initial up-step. The expressions for the rise vector follow immediately.
\end{proof}

\begin{Remark}
\label{rem:right-grafting-rise-ic}
When applying a right-grafting on $u$, the interval-poset of size 1, the rise vector and the initial forest vector have similar expressions:

\begin{align}
\icinf(u \pright{i} I_2) &= 1 + \icinf(I_2); \\
\rises(u \pright{i} I_2) &= 1 + \rises(I_2); \\
\IC(u \pright{i} I_2) &= \left( 1 + \icinf(I_2), \IC^*(I_2), 0 \right); \\
\risesV(u \pright{i} I_2) &= \left( 1 + \rises(I_2), \risesV^*(I_2), 0 \right).
\end{align}

This will be a fundamental property when we define our involutions. Note also that if $\size(I_2) = 0$, we have $\icinf(u \pright{i} I_2) = \rises(u \pright{i} I_2) = 1$ and $\IC(u \pright{i} I_2) = \risesV(u \pright{i} I_2) = 1$.
\end{Remark}

Now, the only statistic which is left to study through the grafting operations is the distance. Recall that by Proposition~\ref{prop:tamari-inversions}, it is given by the number of Tamari inversions. In the same way as for the $\risesV$ vector, it is more complicated to study on the right grafting in which case, we will restrict ourselves to $\size(I_1) = 1$.

\begin{Proposition}
\label{prop:grafting-distance}
Let $I_1$ and $I_2$ be two interval-posets, and $u$ be the interval-poset of size one. Then
\begin{itemize}
\item $\distance(I_1 \pleft I_2) = \distance(I_1) + \distance(I_2)$,
\item $\distance(u \pright{i} I_2) = \distance(I_2) + \contacts(I_2) - i$.
\end{itemize}
\end{Proposition}

Look for example at Figure~\ref{fig:grafting-rise-left}: the Tamari inversion $(1,3)$ of $I_1$ and $(1,2)$ of $I_2$ are kept through $I_1 \pleft I_2$ and no other Tamari inversion is added. For the right grafting, you can look at Figure~\ref{fig:grafting-rise-right}: the interval-poset $I_2$ only has one Tamari inversion $(1,3)$ and we have $\contacts(I_2) = 2$. You can check that $\distance( u \pright{1} I_2) = 2 = 1 + 2 - 1$, the two Tamari inversions being $(2,4)$ and $(1,4)$.

\begin{proof}
We first prove $\distance(I_1 \pleft I_2) = \distance(I_1) + \distance(I_2)$. The condition for a couple $(a,b)$ to be a Tamari inversion is local: it depends only on the values $a \leq k \leq b$. Thus, because the local structure of $I_1$ and $I_2$ is left unchanged, any Tamari inversion of $I_1$ and $I_2$ is kept in $I_1 \pleft I_2$. Now, suppose that $a \in I_1$ and $b \in I_2$. Let $\alpha$ be the label of minimal value in $I_2$ (which has been shifted by the size of $I_1$). By definition, we have $a < \alpha \leq b$ and $a \trprec \alpha$ in $I_1 \pleft I_2$: $(a,b)$ is not a Tamari inversion.

Now, let $I = u \pright{i} I_2$ with $0 \leq i \leq \contacts(I_2)$ and let us prove that $\distance(I) = \distance(I_2) + \contacts(I_2) - i$. Once again, note that the Tamari inversions of $I_2$ are kept through the right grafting. For the same reason, the only Tamari inversions that could be added are of the form $(1,b)$ with $b \in I_2$. Now, let $b$ be a vertex of $I_2$ which is not a decreasing root. This means there is $a < b$ with $b \trprec_{I_2} a$. In $I$, the interval-poset $I_2$ has been shifted by one and so we have: $1 < a < b$ with $b \trprec_I a$: $(1,b)$ is not a Tamari inversion of $I$. Let $b$ be a decreasing root of $I_2$. If $b \trprec_{I} 1$ then $(1,b)$ is not a Tamari inversion. If $b \ntrprec_{I} 1$, we have that: by construction, there is no $a \in I_2$ with $1 \trprec_I a$; because $b$ is a decreasing root there is no $a \in I_2$ with $a < b$ and $b \trprec a$. In other words, $(1,b)$ is a Tamari inversion of $I$ if and only if $b$ is a decreasing root of $I_2$ and $b \ntrprec_I 1$. By the definition of $\pright{i}$ there are exactly $\contacts(I_2) - i$ such vertices.
\end{proof}

\subsection{Grafting trees}
\label{sec:decomposition} 

\begin{Proposition}
\label{prop:decomposition}
An interval-poset $I$ of size $n \geq 1$ is fully determined by a unique triplet $(I_L, I_R, r)$ with $0 \leq r \leq \contacts(I_R)$ and $\size(I_L) + size(I_R) + 1 = n$ such that $I = I_L \pleft u \pright{r} I_R$ with $u$ the unique interval-poset of size 1. We call this triplet the \emph{grafting decomposition} of $I$. See an example on Figure~\ref{fig:grafting-decomposition}.
\end{Proposition}

\begin{Remark}~
\begin{itemize}
\item It can easily be checked that the operation $I_L \pleft u \pright{r} I_R$ is well defined as we have $(I_L \pleft u) \pright{r} I_R = I_L \pleft (u \pright{r} I_R)$. Indeed  $I_L \pleft u$ adds increasing relations from $I_L$ to $u$ while $u \pright{r} I_R$ adds decreasing relation from $I_R$ to $u$. The two operations are independent of each other, see an example on Figure~\ref{fig:grafting-decomposition}. In practice, we think of it as $I_L \pleft (u \pright{r} I_R)$.
\item One or both of the intervals in the decomposition can be empty (of size 0). In particular, the decomposition of $u$ is the triplet $(\emptyset, \emptyset, 0)$.
\end{itemize}
\end{Remark}

\begin{figure}[ht]
\input{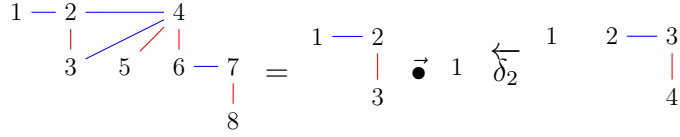}
\caption{Grafting decomposition of an interval-poset}
\label{fig:grafting-decomposition}
\end{figure}

\begin{proof}
This is only a reformulation of \cite[Prop. 3.7]{IntervalPosetsInitial}. Indeed, it was proved that each interval-poset $I$ of size $n$ uniquely appeared in one \emph{composition} $\mathbb{B}(I_L, I_R)$ of two interval-posets where we had $\size(I_L) + \size(I_R) + 1 = n$ and 
\begin{align*}
\mathbb{B}(I_L, I_R) = \sum_{0 \leq i \leq \contacts(I_R)} I_L \pleft u \pright{r} I_R.
\end{align*}

The parameter $r$ identifies which element is $I$ in the composition sum.
\end{proof}

\begin{Definition}
\label{def:grafting-tree}
Let $T$ be a binary tree of size $n$. We write $v_1, \dots, v_n$ the nodes of $T$ taken in in-order (following the binary search tree labeling). Let $\ell : \lbrace v_1, \dots, v_n \rbrace \rightarrow \NN$ be a labeling function on $T$. For all subtrees $T'$ of $T$, we write $\size(T')$ the size of the subtree and $\labels(T') := \sum_{v_i \in T'} \ell(v_i)$ the sum of the labels of its nodes. 

We say that $(T,\ell)$ is a \emph{Tamari interval grafting tree}, or simply \emph{grafting tree} if the labeling $\ell$ satisfies that for every node $v_i$, we have $\ell(v_i) \leq \size(T_R(v_i)) - \labels(T_R(v_i))$ where $T_R(v_i)$ is the right subtree of the node $v_i$.
\end{Definition}

An example is given in Figure~\ref{fig:grafting-tree}: the vertices $v_1, \dots, v_8$ are written in red above the nodes, whereas the labeling $\ell$ is given inside the nodes. For example, you can check the rule on the root $v_4$, we have $\size(T_R(v_4)) - \labels(T_R(v_4)) = 4 - 1 = 3$ and indeed $\ell(v_4) = 2 \leq 3$. The rule is satisfied on all nodes. Note that if the right subtree of a node is empty (which is the case for $v_1$, $v_3$, $v_6$, and $v_8$) then the label is always 0.

\begin{figure}[ht]
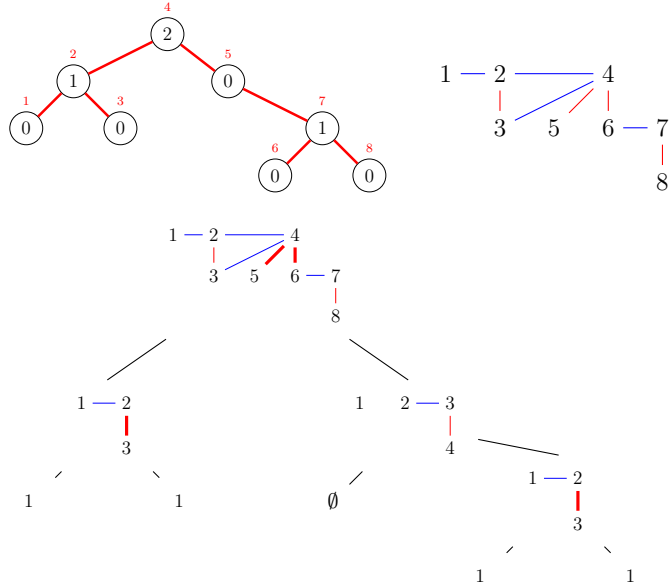

\begin{tabular}{cc}
\scalebox{0.6}{
\input{figures/grafting_tree_example}
}
&
\scalebox{0.8}{\input{figures/interval-posets/I8-ex3}}
\end{tabular}

\input{figures/grafting-tree-construct}

\caption{Example of grafting tree with corresponding interval-poset and grafting decomposition.}
\label{fig:grafting-tree}
\end{figure}

\begin{Proposition}
\label{prop:grafting-tree}
Intervals of the Tamari lattice are in bijection with grafting trees. The grafting tree of an interval-poset $I$ is written as~$\graftingTree(I)$. We compute $\graftingTree(I) = (T, \ell)$ recursively as follows
\begin{itemize}
\item if $I = \emptyset$, then $T$ is the empty binary tree;
\item if $\size(I) > 0$ and $(I_L, I_R, r)$ is the grafting decomposition of $I$, such that $\graftingTree(I_L) = (T_L, \ell_L)$ and $\graftingTree(I_R) = (T_R, \ell_R)$, then $T = \bullet(T_L, T_R)$ and $\ell$ is constructed by keeping unchanged the labels of $T_L$ and $T_R$ given by $\ell_L$ and $\ell_R$ and for the new root $v$ of $T$, $\ell(v) = r$.
\end{itemize}
Besides, 
\begin{align}
\label{eq:grafting-tree-contact}
\contacts(I) = \size(\graftingTree(I)) - \labels(\graftingTree(I)).
\end{align}
\end{Proposition}

Figure~\ref{fig:grafting-tree} illustrates the bijection with the full recursive decomposition. The interval-poset decomposes into the triplet $(I_L, I_R, 2)$ as shown in Figure~\ref{fig:grafting-decomposition}. The left and right subtrees of the grafting tree are obtained recursively by applying the decomposition on $I_L$ and $I_R$. As $\size(I_L) = 3$, the root of $T$ is $v_4$ and we have $\ell(v_4) = 2$, which is indeed the parameter $r$ of the grafting decomposition and also the number of decreasing children of $4$ in the interval-poset.

\begin{proof}
First, let us check that we can obtain an interval-poset from a grafting tree. We read the grafting tree as an expression tree where each empty subtree is replaced by an entry as an empty interval-poset and each node corresponds to the operation $I_L \pleft u \pright{r} I_R$ where $r$ is the label of the node, $I_L$ and $I_R$ the respective results of the expressions of the left and right subtrees, and $u$ the interval poset of size 1. In other words, the interval-poset $I = \graftingTree^{-1}(T, \ell)$ where $(T, \ell)$ is a grafting tree is computed recursively by 
\begin{itemize}
\item if $T$ is empty then $I = \emptyset$;
\item if $T = v_k(T_L, T_R)$ then $I = \graftingTree^{-1}(T_L, \ell_L) \pleft u \pright{r} \graftingTree^{-1}(T_R, \ell_R)$ with $\ell(v_k) = r$, and $\ell_L$ and $\ell_R$ the labeling function $\ell$ restricted to respectively $T_L$ and $T_R$.
\end{itemize}  
We need to check that the operation $u \pright{r} \graftingTree^{-1}(T_R, \ell_R)$ is well-defined, \emph{i.e}, in the case where $T$ is not empty, that we have $0 \leq r \leq \contacts(\graftingTree^{-1}(T_R, \ell_R))$. We do that by also proving by induction that $\contacts(\graftingTree^{-1}(T, \ell)) = \size(T) - \labels(T)$. This is true in the initial case where $T$ is empty: $\contacts(\emptyset) = 0$. Now, suppose that $T = v_k(T_L, T_R)$ with $\ell(v_k) = r$ and that the property is satisfied on $(T_L, \ell_L)$ and $(T_R, \ell_R)$. We write $I_L = \graftingTree^{-1}(T_L, \ell_L)$ and $I_R = \graftingTree^{-1}(T_R,\ell_R)$. In this case, $I' := u \pright{r} I_R$ is well-defined because we have by definition that $r \leq \size(T_R) - \labels(T_R)$, which by induction is $\contacts(I_R)$. Besides, by Proposition~\ref{prop:grafting-contact}, we have $\contacts(I') = 1 + \contacts(I_R) - r$. We now compute $I = I_L \pleft I'$ and we get $\contacts(I) = \contacts(I_L) + 1 + \contacts(I_R) - r$, which is by induction $\size(T_L) - \labels(T_L) + 1 + \size(T_R) - \labels(T_R) - r = \size(T) - \labels(T)$.

Conversely, it is clear from Proposition~\ref{prop:decomposition} that the grafting decomposition of an interval-poset $I$ gives a labeled binary tree $(T, \ell)$. By the unicity of the decomposition, it is is the only labeled binary tree such that $I = \graftingTree^{-1}(T, \ell)$. This proves that $\graftingTree^{-1}$ is injective. To prove that it is surjective, we need to show that $\graftingTree(I)$ is indeed a grafting tree, \emph{i.e.}, the condition on the labels holds. Once again, this is done inductively. An empty interval-poset gives an empty tree and the condition holds. Now if $I$ decomposes into the triplet $(I_L, I_R, r)$ we suppose that the condition holds on $(T_L, \ell_L) = \graftingTree(I_L)$ and $(T_R,\ell_R) = \graftingTree(I_R)$. We know that $0 \leq r \leq \contacts(I_R)$ and we have just proved that $\contacts(I_R)$ is indeed $\size(T_R) - \labels(T_R)$.
\end{proof}

\begin{Proposition}
\label{prop:grafting-direct}
Let $I$ be an interval-poset and $\graftingTree(I) = (T,\ell)$, then
\begin{enumerate}
\item $T$ is the upper bound binary tree of $I$;
\item $\ell(v_i)$ is the number of decreasing children of the vertex $i$ in $I$.
\end{enumerate} 
\end{Proposition}

In other words, the grafting tree of an interval-poset can be obtained directly without using the recursive decomposition. Also, the tree $T$ only depends on the initial forest and the labeling $\ell$ only depends on the final forest.

\begin{proof}
We prove the result by induction on the size of $I$. If $I$ is empty, there is nothing to prove. We then suppose that $I$ decomposes into a triplet $(I_L, I_R, r)$ with $k = \size(I_L) + 1$. We suppose by induction that the proposition is true on $I_L$ and $I_R$. Let $(T, \ell) = \graftingTree(I)$, $(T_L, \ell_L) = \graftingTree(I_L)$, and $(T_R, \ell_R) = \graftingTree(I_R, \ell_R)$. By induction, $T_L$ and $T_R$ are the upper bound binary trees of $I_L$ and $I_R$ respectively. In \cite[Prop. 3.4]{IntervalPosetsInitial}, we proved $T = v_k(T_L, T_R)$, which by construction of the initial forest is indeed the upper bound binary tree of $I$. The result on the labeling function $\ell$ is obtained by induction on $\ell_L$ and $\ell_R$ for all vertices $v_i$ with $i \neq k$. Besides, by definition of the grafting tree, we have $\ell(v_k) = r$, which is indeed the number of decreasing children of the vertex $k$ in $I$ by the definition of the right grafting $\pright{r}$.
\end{proof}

\begin{Remark}
Note that the grafting tree of an Tamari interval has similarities with another structure in bijection with interval-posets: \emph{closed flow} on a planar forest, which was described in \cite{ME_FPSAC2014}. The planar forest associated to an interval-poset depends only on the initial forest of the interval, \emph{i.e.}, only on its upper bound binary tree, which also gives the shape of the grafting tree. In other words, given a binary tree~$T$, there is a one-to-one correspondence between the possible labeling $\ell$ such that $(T,\ell)$ is a grafting tree and the closed flows on a certain planar forest $F$. As described in \cite[Fig. 10]{ME_FPSAC2014}, the forest $F$ corresponding to $T$ is obtained by a classical bijection often referred to by the ``left child to left brother'' bijection. It consists of transforming, for each node of the binary tree, the left child into a left brother and the right child into the last child in the planar forest. Now, the flow itself depends on the decreasing forest of the interval-poset just as the labeling $\ell$ of the grafting tree. Each node receiving a $-1$ in the flow corresponds to a node with a positive label in the grafting tree. 
\end{Remark}

\begin{Remark}
The ``left child to left brother`` bijection to planar forest also gives a direct bijection between grafting trees and $(1,1)$ description trees of \cite{CoriSchaefferDescTrees} (the planar forest is turned into a tree by adding a root). The labels $\ell'$ of the $(1,1)$ description trees are obtained through a simple transformation from $\ell$: for each node $v$, $\ell'(v) = 1 + \size(T_R(v)) - \labels(T_R(v)) - \ell(v)$. 
\end{Remark}

\begin{Proposition}
\label{prop:grafting-tree-contact}
Let $I$ be an interval-poset and $(T,\ell) = \graftingTree(I)$ its grafting tree with $v_1, \dots, v_n$ the vertices of~$T$. Then $\contactsV^*(I) = (\ell(v_1), \dots, \ell(v_{n-1}))$.
\end{Proposition}

\begin{proof}
Remember that $\contactsV^*(I) = \DC^*(I)$ by Proposition~\ref{prop:contact-dc}, \emph{i.e.}, the final forest vector given by reading the number of decreasing children of the vertices in $I$. Then the result is a direct consequence of Proposition~\ref{prop:grafting-direct}.
\end{proof}

\begin{Proposition}
\label{prop:grafting-tree-distance}
Let $I$ be an interval-poset and $(T,\ell) = \graftingTree(I)$ its grafting tree with $v_1, \dots, v_n$ the vertices of $T$. Let $d_i = \size(T_R(v_i)) - \labels(T_R(v_i)) - \ell(v_i)$ for all $1 \leq i \leq n$ where $T_R(v_i)$ is the right subtree of the vertex $v_i$ in $T$. Then
\begin{align}
\distance(I) = \sum_{1 \leq i \leq n} d_i.
\end{align}
\end{Proposition}

For example, on Figure~\ref{fig:grafting-tree}, we have all $d_i = 0$ except for $d_4 = 4 - 1 - 2 = 1$ and $d_5 = 3 - 1 = 2$. This indeed is consistent with $\distance(I) = 3$, the 3 Tamari inversions being $(4,7)$, $(5,6)$, and $(5,7)$. More precisely, the number $d_i$ is the number of Tamari inversions of the form~$(i,*)$.

\begin{proof}
Once again, we prove the property inductively. This is true for an empty tree where we have $\distance(I) = 0$. Now, let $I$ be a non-empty interval-poset, then $I$ decomposes into a triplet $(I_L, I_R, r)$ with $I = I_L \pleft u \pright{r} I_R$. Proposition~\ref{prop:grafting-distance} gives us
\begin{align}
\distance(I) &= \distance(I_L \pleft u \pright{r} I_R) \\
&= \distance(I_L) + \distance(u \pright{r} I_R) \\
&= \distance(I_L) + \distance(I_R) + \contacts(I_R) - r.
\end{align}

Now let $(T, \ell) = \graftingTree(I)$. By definition, we have $T = v_k(T_L, T_R)$ with $k = \size(T_L) + 1$, $(T_L, \ell_L) = \graftingTree(I_L)$, and $(T_R, \ell_R) = \graftingTree(I_R)$. Using the induction hypothesis and \eqref{eq:grafting-tree-contact}, we obtain

\begin{align}
\sum_{1 \leq i \leq n} d_i &= \sum_{1 \leq i < k} d_i + d_k + \sum_{k < i \leq n} d_i  \\
&= \distance(I_L) + \size(T_R) - \labels(T_R) - \ell(v_k) + \distance(I_R) \\
&= \distance(I_L) + \contacts(I_R) - r + \distance(I_R).
\end{align}
\end{proof}

\subsection{Left branch involution on the grafting tree}
\label{sec:grafting-tree-involution}

We now give an interesting involution on the grafting tree, which in turns gives an involution on Tamari intervals. In Section~\ref{sec:statement}, we mentioned that the rise-contact involution on Dyck paths (not intervals) used the reversal of a Dyck path conjugated with the Tamari symmetry. The equivalent of the Dyck path reversal on the corresponding binary tree is also a classical involution, which we call the \emph{left branch involution}. Applying this involution on grafting trees will allow us to generalize it to intervals. A \emph{right hanging binary tree} is a binary tree whose left subtree is empty. An alternative way to see a binary tree is to understand it as list of right hanging binary trees grafted together on its left-most branch. For example, the tree of Figure~\ref{fig:grafting-tree} can be decomposed into 3 right hanging binary trees : the one with vertex $v_1$, the one with vertices $v_2$ and $v_3$ and the one with vertices $v_4$ to $v_8$.

\begin{Definition}
The \emph{left branch involution} on binary trees is the operation that recursively reverse the order of right hanging trees on every left branch of the binary tree.

We write $\leftbranch(T)$ the image of a binary tree $T$ through the involution.
\end{Definition}

This operation is a very classical involution on binary trees, see Figure~\ref{fig:left-branch-involution} for an example. It is implemented in SageMath \cite{SageMath2017} as the \texttt{left\_border\_symmetry} method on binary trees. You can also understand it in a recursive manner: if $T$ is an non-empty tree with $T_L$ and $T_R$ as respectively left and right subtrees, then the image of $T$ can be constructed from the respective image $T_R'$ and $T_L'$ of $T_R$ and $T_L$ following the structure of Figure~\ref{fig:left-branch-involution-recursive}. The root is grafted on the left-most branch of $T_L'$ with an empty left subtree and $T_R'$ as a right subtree.

\begin{figure}[ht]
\input{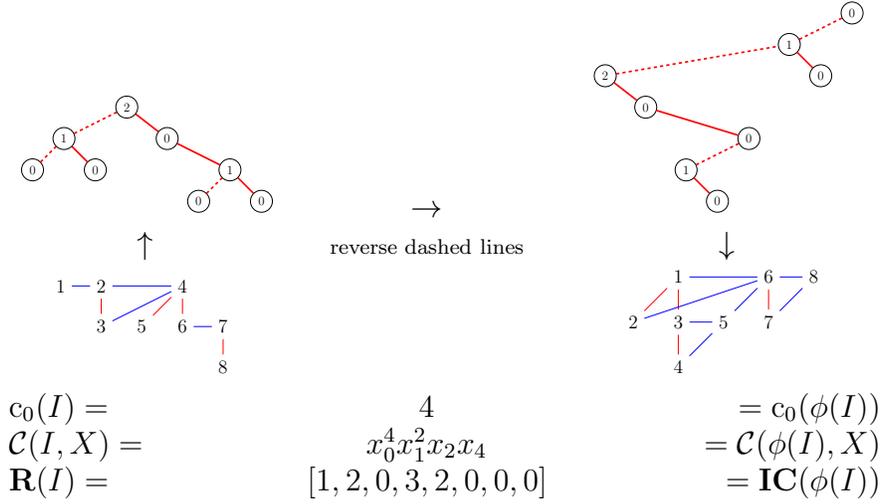}
\caption{The left-branch involution.}
\label{fig:left-branch-involution}
\end{figure}

\begin{figure}[ht]
\input{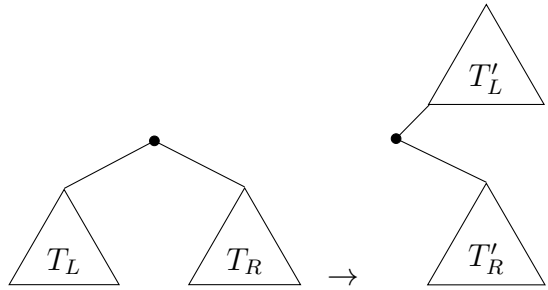}
\caption{The left-branch involution seen recursively.}
\label{fig:left-branch-involution-recursive}
\end{figure}

\begin{Proposition}
The left branch involution is an involution on grafting trees.
\end{Proposition}

\begin{proof}
First, let us clarify what the involution means on a grafting tree~$(T, \ell)$: we apply the involution on the binary tree $T$ and the vertices \emph{move along with their labels} as illustrated in Figure~\ref{fig:left-branch-involution}. We obtain a new labeled binary tree $(T', \ell')$ where every vertex $v_i$ of $T$ is sent to a new vertex $v_{i'}$ of $T'$ such that $\ell(v_i) = \ell'(v_{i'})$. For example, in Figure~\ref{fig:left-branch-involution}, the root $v_4$ of $T$ is sent to $v_1$ of $T'$, with $\ell(v_4) = \ell'(v_1) = 2$.

The only thing to check is that $\ell'$ still satisfies the grafting tree condition. This is immediate. Indeed, for $v_i \in T$, and $T_R(v_i)$ its right subtree, we have $\ell(v_i) \leq \size(T_R(v_i)) - \labels(T_R(v_i))$. Now, if $v_{i'}$ is the image of $v_i$ and $T'_R(v_{i'})$ its right subtree, even though $T'_R$ might be different from $T_R$, the statistics are preserved: $\size(T'_R(v_{i'})) = \size(T_R(v_i))$ and $\labels(T'_R(v_{i'})) = \labels(T_R(v_i))$, because the involution only acts on left branches and the set of labels of the right subtree is preserved.
\end{proof}

As a consequence, we now have an involution on Tamari intervals.

\begin{Definition}[The Left Branch Involution]
\label{def:leftbranch-intervals}
The left branch involution on Tamari intervals is defined by the left branch involution on their grafting trees.
\begin{align}
\leftbranch(I) := \graftingTree^{-1}(\leftbranch(\graftingTree(I)))
\end{align}
\end{Definition}

The grafting tree seems to be the most natural object to describe the involution. Indeed, even though it can be easily computed on interval-posets using decomposition and graftings, we have not seen any simple direct description of it. Furthermore, if we understand the interval as a couple of a lower bound and upper bound, then the action on the upper bound is simple: the shape of the upper bound binary tree is given by the grafting tree and so the involution on the upper bound is only the classical left-branch involution, which corresponds to reversing the Dyck path. Nevertheless, the action on the lower bound cannot be described as an involution on binary trees: it depends on the corresponding upper bound. One way to understand this involution is that we apply the left-branch involution on the upper bound binary tree and the lower bounds ``follows'' in the sense given by the labels of the grafting tree.

\begin{Proposition}
\label{prop:leftbranch-statistics}
Let $I$ be an interval of Tamari, then
\begin{align}
\label{eq:leftbranch-contacts}
\contacts(I) &= \contacts(\leftbranch(I)); \\
\label{eq:leftbranch-contactsP}
\contactsP(I) &= \contactsP(\leftbranch(I)); \\
\label{eq:leftbranch-distance}
\distance(I) &= \distance(\leftbranch(I)); \\
\label{eq:leftbranch-rises}
\risesV(I) &= \IC(\leftbranch(I)).
\end{align}

In other words, the involution exchanges the rise vector and initial forest vector while leaving unchanged the number of contacts, the contact monomial, and the distance.
\end{Proposition}

\begin{proof}
Points \eqref{eq:leftbranch-contacts} and \eqref{eq:leftbranch-contactsP} are immediate. Indeed, \eqref{eq:grafting-tree-contact} tells us that $\contacts(I)$ is given by $\size(\graftingTree(I)) - \labels(\graftingTree(I))$ : this statistic is not changed by the involution. Now remember that, by Proposition~\ref{prop:grafting-contact}, the values $\contactsStep{1}(I), \dots, \contactsStep{n}(I)$ are given by $\ell(v_1), \dots, \ell(v_n)$, so $\contactsP(I) = x_{\contacts(I)} x_{\ell(v_1)} \dots x_{\ell(v_{n-1})} = \frac{x_{\contacts(I)} x_{\ell(v_1)} \dots x_{\ell(v_{n})}}{x_0}$. This monomial is commutative and the involution sending $\ell$ to $\ell'$ only applies a permutation on the indices: the monomial itself is not changed. Also, we always have $\ell(v_n) = \ell'(v_n) = 0$ so the division by $x_0$ is still possible after the permutation and still removes the last value $x_{\ell'(v_n)}$. As an example, on Figure~\ref{fig:left-branch-involution}, we have $\contactsP(I) = x_4 x_0 x_1 x_0 x_2 x_0 x_0 x_1 = x_0^4 x_1^2 x_2 x_4 = x_4 x_2 x_0 x_1 x_0 x_0 x_1 x_0 x_0 = \contactsP(\leftbranch(I))$.

Point \eqref{eq:leftbranch-distance} is also immediate by Proposition~\ref{prop:grafting-tree-distance}. Indeed, for all $1 \leq i \leq n$, we have $d_i = \size(T_R(v_i)) - \labels(T_R(v_i)) = \size(T_R(v_{i'})) - \labels(T_R(v_{i'})) = d_{i'}$ if $v_i$ is sent to $v_{i'}$ by the involution. Once again, the values $d_1, \dots, d_n$ are only permuted and the sum stays the same.

We prove point~\eqref{eq:leftbranch-rises} by induction on the size of the tree. It is trivially true when $\size(I) = 0$ (both vectors are empty). Now suppose that $I$ is an interval-poset of size $n > 0$. Let $(T, \ell) = \graftingTree(I)$, then $T$ is a non-empty binary tree with two subtrees $T_L$ and $T_R$ (which can be empty) and a root node $v$ such that $\ell(v) = i$. Let us call $I_L$ and $I_R$ the interval-posets corresponding to $T_L$ and $T_R$ respectively. By definition, we have that

\begin{align}
I = I_L \pleft u \pright{i} I_R.
\end{align}

We call $T_L'$ and $T_R'$ the respective image of $T_L$ and $T_R$ through the left branch involution and $I_L'$ and $I_R'$ the corresponding interval-posets. As both $T_L$ and $T_R$ are of size strictly less than $n$, we have by induction that

\begin{align}
\risesV(I_L) &= \IC(I_L'), \\
\nonumber
\risesV(I_R) &= \IC(I_R').
\end{align}

Following the recursive description of the left branch involution given on Figure~\ref{fig:left-branch-involution-recursive}, we obtain that
the image $I' := \leftbranch(I)$ is given by

\begin{align}
\label{eq:leftbranch-intervals-recursive}
\left( u \pright{i} I_R' \right) \pleft I_L'.
\end{align}

We are using a small shortcut here as this expression does not exactly correspond to the definition of the grafting tree. Indeed, $T_L'$ is a whole tree, not a single node. Nevertheless, it can be easily checked that the left product $\pleft$ is associative. Then any tree can be seen as a series of a right-hanging trees grafted to each other as in the following picture.

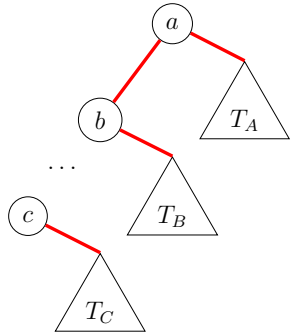
\begin{center}
\scalebox{.8}{\input{figures/right-hanging-trees}}
\end{center}

The definition gives us that the interval-poset is computed by

\begin{align}
&(\dots((u \pright{c} I_C) \pleft \dots (u \pright{b} I_B)) \pleft (u \pright{a} I_A))\\
\nonumber
 &= (u \pright{c} I_C) \pleft \dots  (u \pright{b} I_B) \pleft (u \pright{a} I_A).
\end{align}

Using \eqref{eq:leftbranch-intervals-recursive}, we obtain the desired result. Indeed, let $J = u \pright{i} I_R$ and $J' = u \pright{i} I_R'$. If $I_R$ is empty, so is $I_R'$ and we have $\risesV(J) = \IC(J') = (1)$. If not, we use Propositions~\ref{prop:grafting-ic} and~\ref{prop:grafting-rise-right} and Remark~\ref{rem:right-grafting-rise-ic} to obtain
\begin{align}
\risesV(J) &= (1 + \rises(I_R), \risesV^*(I_R), 0) \\
&= (1 + \icinf(I_R'), \IC^*(I_R'), 0) \\
&= \IC(J').
\end{align}

Now by using Propositions~\ref{prop:grafting-rise-left} and~\ref{prop:grafting-ic}, we obtain

\begin{align}
\risesV(I) = \risesV(I_L \pleft J) 
&= (\risesV(I_L), \risesV(J)) \\
&= (\IC(I_L'), \IC(J')) \\
&= \IC(J' \pleft I_L') = \IC(I').
\end{align}

\end{proof}

\subsection{The complement involution and rise-contact involution}

As we have seen in Section~\ref{sec:statement}, the rise-contact involution on Dyck paths is a conjugation of the Tamari symmetry involution by the Dyck path reversal involution. The equivalent of the Dyck path reversal on intervals is the left-branch involution on the grafting tree. We now need to describe what is the Tamari symmetry on intervals: this is easy, especially described on interval-posets.

\begin{Definition}[The Complement Involution]
The \emph{complement} of an interval-poset $I$ of size $n$ is the interval-poset $J$ defined by
\begin{align}
i \trprec_J j \Leftrightarrow (n+1 -i) \trprec_I (n+1 -j). 
\end{align}
We write $\compl(I)$ the complement of $I$.
\end{Definition} 

\begin{figure}[ht]
\input{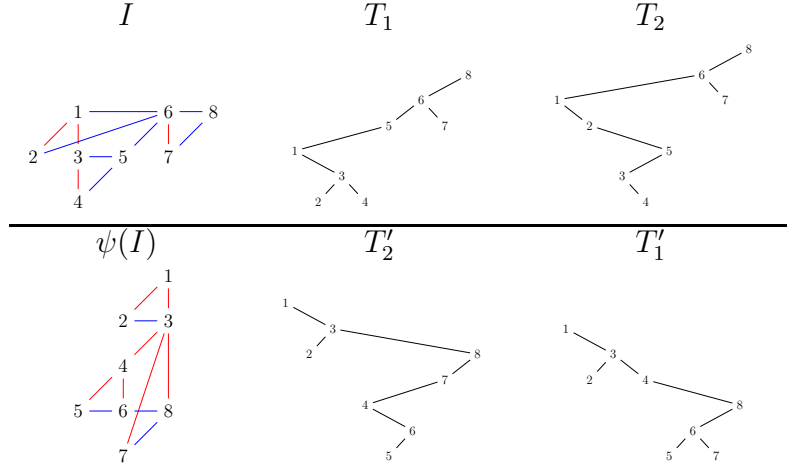}
\caption{The complement of an interval-poset}
\label{fig:compl}
\end{figure}

An example is shown on Figure~\ref{fig:compl}. It is clear by Definition~\ref{def:interval-poset} that this is still an interval-poset. Basically, this is an involution exchanging increasing and decreasing relations. This corresponds to the up-down symmetry of the Tamari lattice. It is a well known fact that the Tamari lattice is isomorphic to its inverse by sending every tree $T$ to its reverse $T'$ where the left and right subtrees have been exchanged on every node. Let $T_1$ and $T_2$ be respectively the lower and upper bounds of an interval~$I$. Let $T_1'$ and $T_2'$ be the respective reverses of $T_1$ and $T_2$. Then $T_1'$ is the upper bound of $\compl(I)$ and $T_2'$ is the lower bound.

\begin{Proposition}
\label{prop:complement-ic-dc}
Let $I$ be an interval-poset, then $\IC(I) = \DC(\compl(I))$.
\end{Proposition}

\begin{proof}
Every increasing relation $a \trprec_I b$ is sent to a decreasing relation $(n+1 - a) \trprec_{\compl(I)} (n+1 - b)$. In particular, each connected component of the initial forest of $I$ is sent to exactly one connected component of the final forest of $\compl(I)$ and so $\icinf(I) = \dcstep{0}(\compl(I))$. Now, if a vertex $b$ has $k$ increasing children in $I$, its image $(n+1-b)$ has $k$ decreasing children in $\compl(I)$ so $\icstep{b}(I) = \dcstep{n+1-b}(\compl(I))$. Remember that $\IC^*$ reads the numbers of increasing children in reverse order from $n$ to $2$ whereas $\DC^*$ reads the number of decreasing children in the natural order from $1 = n+1-n$ to $n-1 = n+1 -2$. We conclude that $\IC(I) = \DC(\compl(I))$. 
\end{proof}

\begin{Proposition}
\label{prop:complement-distance}
Let $I$ be an interval-poset, then $\distance(I) = \distance(\compl(I))$.

More precisely, $(a,b)$ is a Tamari inversion of $I$ if and only if $(n+1-b, n+1-a)$ is a Tamari inversion of $\compl(I)$. 
\end{Proposition}

\begin{proof}
Let $a < b$ be two vertices of $I$, we set $a' = n+1-b$ and $b' = n+1-a$. 
\begin{itemize}
\item There is $a \leq k <b$ with $b \trprec_I k$ if and only if there is $k' = n+1-k$ with $a' < k' \leq b'$ and $a' \trprec_{\compl(I)} k'$.
\item There is $a < k \leq b$ with $a \trprec_I k$ if and only if there is $k' = n+1-k$ with $a' \leq k' < b'$ and $b' \trprec_{\compl(I)} k'$.
\end{itemize}
In other words, $(a,b)$ is a Tamari inversion of $I$ if and only if $(a',b')$ is a Tamari inversion of $\compl(I)$. By Proposition~\ref{prop:tamari-inversions}, this gives us $\distance(I) = \distance(\compl(I))$.
\end{proof}

You can check on Figure~\ref{fig:compl} that $I$ has 3 Tamari inversions $(1,5)$, $(2,3)$, and $(2,5)$, which give respectively the Tamari inversions $(4,8)$, $(6,7)$, and $(4,7)$ in $\compl(I)$.  We are now able to state the following Theorem, which gives an explicit combinatorial proof of Theorem~\ref{thm:main-result-classical}. We give an example computation on Figure~\ref{fig:rise-contact-involution}. You can run more examples and compute tables for all intervals using the provided live {\tt Sage-Jupyter} notebook~\cite{JNotebook}.

\begin{figure}[ht]
\input{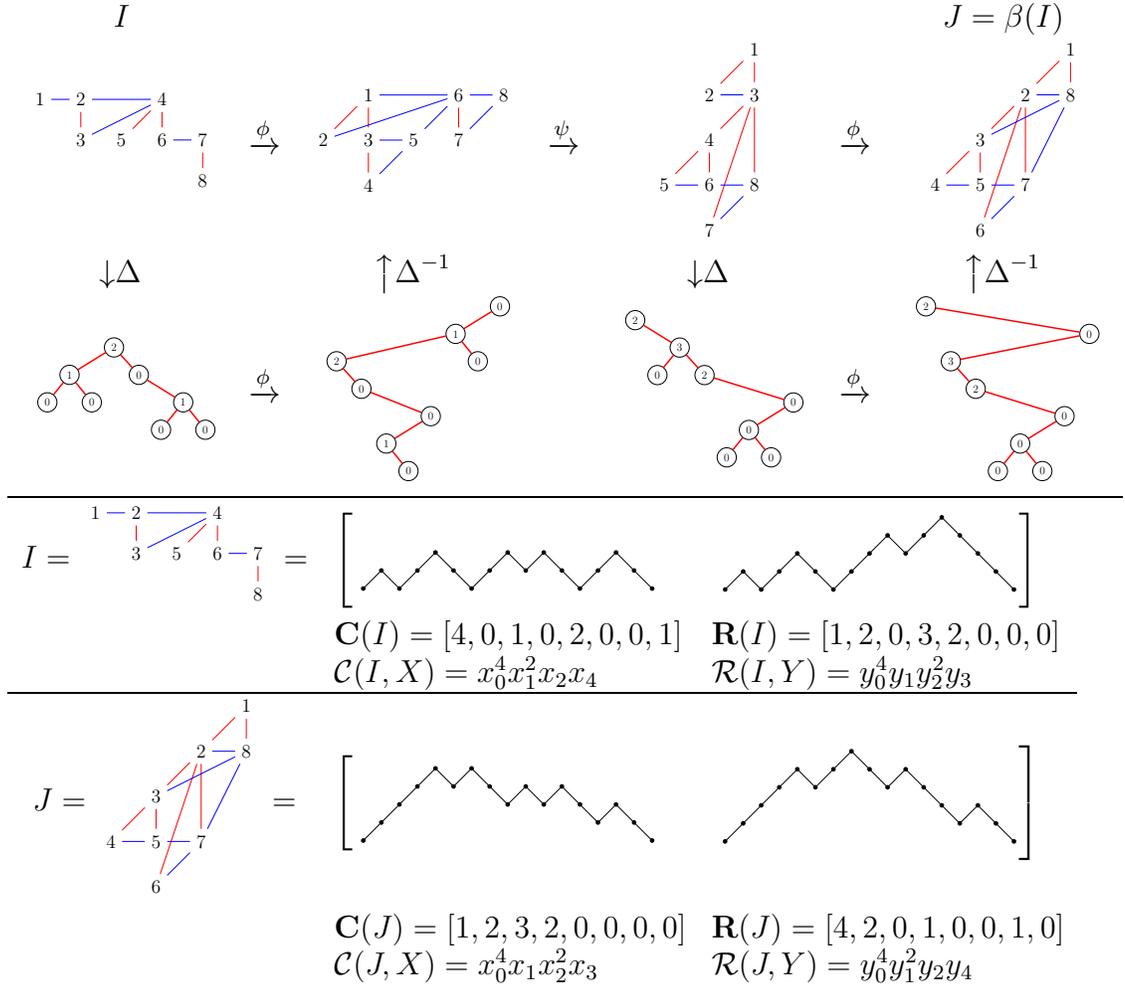}
\caption{The rise-contact involution on an example.}
\label{fig:rise-contact-involution}
\end{figure}

\begin{Theorem}[the rise-contact involution]
\label{thm:rise-contact-statistics}
Let $\risecontact$ be the \emph{rise-contact involution} defined by
\begin{align}
\risecontact = \leftbranch \circ \compl \circ \leftbranch.
\end{align}
Then $\risecontact$ is an involution on Tamari intervals such that, for an interval~$I$ and a commutative alphabet $X$,
\begin{align}
\label{eq:rise-contact-rises}
\rises(I) &= \contacts(\risecontact(I)); \\
\label{eq:rise-contact-partitions}
\risesP(I,X) &= \contactsP(\risecontact(I),X); \\
\label{eq:rise-contact-distance}
\distance(I) &= \distance(\risecontact(I)).
\end{align}
\end{Theorem}

\begin{proof}
The operation $\risecontact$ is clearly an involution because it is the conjugate of the complement involution $\compl$ by the left-branch involution~$\leftbranch$. We obtain \eqref{eq:rise-contact-distance} immediately as the distance is constant through $\leftbranch$ by~\eqref{eq:leftbranch-distance} and through $\compl$ by Proposition~\ref{prop:complement-distance}. Now, using Propositions~\ref{prop:leftbranch-statistics} and \ref{prop:complement-ic-dc},  we have
\begin{align}
\contacts(\risecontact(I)) &= \contacts(\leftbranch \circ \compl \circ \leftbranch (I)) = \contacts(\leftbranch \circ \compl(I)) = \dcstep{0}(\leftbranch \circ \compl(I)) \\
&= \icinf(\leftbranch(I)) = \rises(I),
\end{align}
which proves \eqref{eq:rise-contact-rises}. Now, by Proposition~\ref{prop:leftbranch-statistics}, we have that $\contactsV(\leftbranch \circ \compl(I))$ is a permutation of $\contactsV(\leftbranch \circ \compl \circ \leftbranch (I))$. We then use Proposition~\ref{prop:complement-ic-dc} and again Proposition~\ref{prop:leftbranch-statistics}
\begin{align}
\contactsV(\leftbranch \circ \compl(I)) &= \DC (\leftbranch \circ \compl(I)) = \IC(\leftbranch (I)) = \risesV(I).
\end{align}
This means that $\risesV(I)$ is a permutation of $\contactsV(\beta(I))$, and so, because $X$ is a commutative alphabet, \eqref{eq:rise-contact-partitions} holds.
\end{proof}

\begin{Remark}
The reader might notice at this point that the notion of Tamari interval-poset is not completely necessary to the definition of the rise-contact involution. Indeed, one novelty of this paper is the introduction of the grafting tree, which, we believe, truly encapsulates the recursive structure of the Tamari intervals. As an example, it is an interesting (and easy) exercise to recover the functional equation first described in \cite{Chap} and later discussed in \cite{IntervalPosetsInitial} using solely grafting trees. Nevertheless, please note that the rise-contact involution cannot be described using solely grafting trees. Indeed, grafting trees are the natural object to apply the left-branch involution but they do not behave nicely through the complement involution. In this case, the interval-posets turn out to be the most convenient object. The complement involution can also be described directly on intervals of binary trees but then it makes it more difficult to follow some statistics such as the distance. For these reasons, and also for convenience and reference to previous results, we have kept interval-posets central in this paper.
\end{Remark}

\begin{Remark}
In \cite{DecompBeta10} and \cite{InvolutionBeta10}, the authors describe an interesting involution on $(1,0)$ description trees that leads to the equi-distribution of certain statistics. Their bijection is described recursively through grafting and up-raising of trees. Some similar operations can be defined on $(1,1)$ description trees. An interesting question is then: is there a a direct description of the rise-contact involution on $(1,1)$ description trees? The answer is most probably \emph{yes}. Actually, this resumes to understanding the complement involution on $(1,1)$ description trees. We leave that for further research or curious readers. 
\end{Remark}

\section{The $m$-Tamari case}
\label{sec:mtam}

\subsection{Definition and statement of the generalized result}
\label{sec:mtam-def}

The $m$-Tamari lattices are a generalization of the Tamari lattice where objects have an $(m+1)$-ary structure instead of binary. They were introduced in \cite{BergmTamari} and can be described in terms of $m$-ballot paths. An $m$-ballot path is a lattice path from $(0,0)$ to $(nm,n)$ made from horizontal steps $(1,0)$ and vertical steps $(0,1)$, which always stays above the line $y=\frac{x}{m}$. When $m=1$, an $m$-ballot path is just a Dyck path where up-steps and down-steps have been replaced by respectively vertical steps and horizontal steps. They are well known combinatorial objects counted by the $m$-Catalan numbers
\begin{equation}
\frac{1}{mn + 1} \binom{(m+1)n}{n}.
\end{equation}

They can also be interpreted as words on a binary alphabet and the notion of \emph{primitive path} still holds. Indeed, a primitive path is an $m$-ballot path which does not touch the line $y=\frac{x}{m}$ outside its end points. From this, the definition of the rotation on Dyck path given in Section~\ref{sec:tamari} can be naturally extended to $m$-ballot-paths, see Figure~\ref{fig:mpath-rot}.

\begin{figure}[ht]
\centering
\input{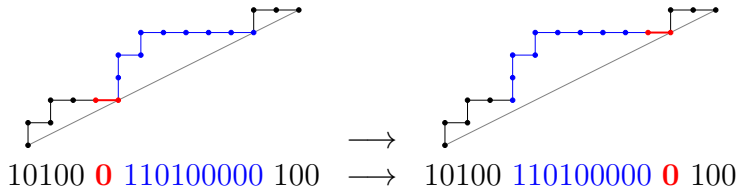}
\caption{Rotation on $m$-ballot paths.}
\label{fig:mpath-rot}
\end{figure}

When interpreted as a cover relation, the rotation on $m$-ballot paths induces a well-defined order, which is a lattice \cite{BergmTamari}. This is what we call the $m$-Tamari lattice or $\Tamnm$, see Figure \ref{fig:mTamari} for an example.

\begin{figure}[ht]

  \input{figures/mTamari-3-3-paths}
  
  \caption{$m$-Tamari on $m$-ballot paths: $\Tam{3}{2}$.}
  
  \label{fig:mTamari}
\end{figure}

The intervals of $m$-Tamari lattices have also been studied. In \cite{mTamari}, it was proved that they are counted by
\begin{equation}
\label{eq:m-intervals-formula}
I_{n,m} = \frac{m+1}{n(mn +1)} \binom{(m+1)^2 n + m}{n - 1}.
\end{equation} 
They were also studied in \cite{IntervalPosetsInitial} where it was shown that they are in bijection with some specific families of Tamari interval-posets. Our goal here is to use this characterization to generalize Theorem~\ref{thm:main-result-classical} to intervals of $m$-Tamari, thus proving Conjecture 17 of \cite{PRThesis}. First, let us introduce the $m$-statistics, which correspond to the classical cases statistics defined in Definition~\ref{def:contact-rise-dw}.

\begin{Definition}
\label{def:m-contact-rise-dw}
Let $B$ be an $m$-ballot path. We define the following $m$-\emph{statistics}.
\begin{itemize}
\item $\mcontacts(B)$ is the number of non-final contacts of the path $B$: the number of time the path $B$ touches
the line $y=\frac{x}{m}$ outside the last point.

\item $\mrises(B)$ is the initial rise of $B$: the number of initial consecutive vertical steps.

\item Let $u_i$ be the $i^{th}$ vertical step of $B$, $(a,b)$ the coordinate of its starting point and $j$ an integer such that $1 \leq j \leq m$. We consider the line $\ell_{i,j}$ starting at $(a, b + \frac{j}{m})$ with slope $\frac{1}{m}$ and the portion of path $d_{i,j}$ of $B$ which starts at $(a,b+1)$ and stays above the line $\ell_{i,j}$. From this, we define $\mcontactsStep{i,j}(B)$ the number of non-final contacts between $\ell_{i,j}$ and $d_{i,j}$.

\item Let $v_i$ be the $i^{th}$ horizontal step of $B$, we say that the number of consecutive vertical steps right after $v_i$ are the \emph{$m$-rises} of $v_i$ and write $\mrisesStep{i}(B)$.

\item $\mcontactsV(B) := (\mcontacts(B), \mcontactsStep{1,1}(B), \dots, \mcontactsStep{1,m}(B), \dots,  \mcontactsStep{n,1}(B), \dots, \mcontactsStep{n,m-1}(B))$ is the \emph{$m$-contact vector} of~$B$.

\item $\mrisesV(B) := (\mrises(B), \mrisesStep{1}(B), \dots, \mrisesStep{nm-1}(B))$ is the \emph{$m$-rise vector} of~$B$.

\item Let  $X = (x_0, x_1, x_2, \dots)$ be a commutative alphabet, we write $\mcontactsP(B,X)$ the monomial $x_{v_0}, \dots x_{v_{nm-1}}$ where $\mcontactsV(I) = (v_0, \dots, v_{nm-1})$ and we call it the \emph{$m$-contact monomial} of $B$.

\item Let  $Y = (y_0, y_1, y_2, \dots)$ be a commutative alphabet, we write $\risesP(B,Y)$ the monomial $y_{w_0}, \dots, y_{w_{nm-1}}$ where $\mrisesV(I) = (w_0, \dots, w_{nm-1})$ and we call it the \emph{$m$-rise monomial} of $B$.
\end{itemize}

Besides, we write $\size(B):=n$. An $m$-ballot path of size $n$ has $n$ vertical steps and $nm$ horizontal steps.
\end{Definition}

\begin{figure}[ht]
\input{figures/m-contacts-rises-example}
\caption{The $m$-contacts and $m$-rises of a ballot path.}
\label{fig:m-stat}
\end{figure}
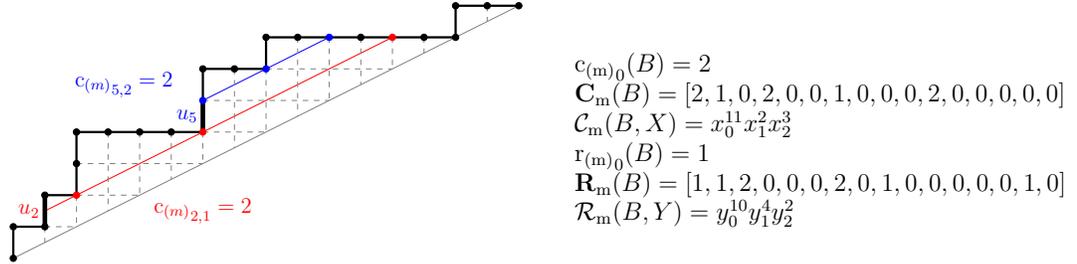

An example is given on Figure~\ref{fig:m-stat}. When $m = 1$, this is the same as Definition~\ref{def:contact-rise-dw}. Note also that we will later define a bijection between $m$-ballot paths and certain families of Dyck paths which also extends to intervals: basically any element of $\Tamnm$ can also be seen as an element of $\Tamntm$ but the statistics are not exactly preserved, which is why we use slightly different notations for $m$-statistics to avoid any confusion.

Both $\mcontactsV(B)$ and $\mrisesV(B)$ are of size $nm$. Also, note that even though $\ell_{i,j}$ does not always starts at an integer point, 
the contacts with the subpath $d_{i,j}$ only happen at integer points. Because the final contact is not counted, it can happen that $\mcontactsStep{i,j} = 0$ even when $d_{i,j}$ is not reduced to a single point. Indeed, the initial point is a contact only when $j = m$. In this case, the definition of $\mcontactsStep{i,m}$ is similar to the classical case from Definition~\ref{def:contact-rise-dw}.

The $m$-rise vector somehow partitions the vertical steps and it is clear that $\sum_{0 \leq i \leq nm} \mrisesStep{i}(B) = n$. Actually, we also have $\sum_{0 \leq i \leq n; 1 \leq j \leq m} \mcontactsStep{i,j}(B) = n$. We see this through another description of the non-zero values of the vector which makes the relation to \cite[Conjecture 17]{PRThesis} explicit.

\begin{Proposition}
For each vertical step $u_i$ of an $m$-ballot path, let $a_i$ be the number of $1 \times 1$ squares that lies horizontally between the step $u_i$ and the line $y = \frac{x}{m}$. This gives us $a(B) = [a_1, \dots, a_n]$, the \emph{area vector} of~$B$. We partition the values of $a(B)$ such that $a_i$ and $a_j$ are in the same set if $a_i = a_j$ and for all $i'$ such that $i \leq i' \leq j$, then $a_{i'} \geq a_i$. Let $\lambda = (\lambda_1 \geq \lambda_2 \geq \dots \geq \lambda_k) $ be the integer partition obtained by keeping only the set sizes and let $e(B,X) = x_{\lambda_1} \dots x_{\lambda_k}$ a monomial on a commutative alphabet $X$. Then $e(B,X) = \mcontactsP(B,X)$ with $x_0 = 1$.
\end{Proposition}

The definition of $e(B,X)$ comes from \cite[Conjecture 17]{PRThesis}. As an example, the area vector of the path from Figure~\ref{fig:m-stat} is $(0,1,2,4,2,4,4,0)$. The set partition is $\lbrace \lbrace a_1, a_8 \rbrace, \lbrace a_2 \rbrace, \lbrace a_3, a_5 \rbrace, \lbrace a_4 \rbrace, \lbrace a_6, a_7 \rbrace \rbrace$. In particular, the area vector always starts with a 0 and each new 0 corresponds to a contact between the path and the line. Here, we get $\lambda = (2,2,2,1,1)$, which indeed gives $e(B,X) = x_1 ^2 x_2^3 = \contactsP(B,X)$ at $x_0 = 1$.

\begin{proof}
If the step $u_i$ starts at a point $(x,y)$, then we have by definition $my = x + a_i$. In particular, if $a_i = a_j$, then $u_i$ and $u_j$ both have a contact with a same affine line $s$ of slope $\frac{1}{m}$. Then $a_i$ and $a_j$ belong to the same set in the partition if and only if the path between $u_i$ and $u_j$ stays above the line $s$. More precisely, the line $s$ cuts a section $p$ of the path, starting at some point $(a, b + \frac{j}{m})$ where $(a,b)$ is the starting point of a vertical step and $1 \leq j \leq m$. The non-final contacts of this path $p$ with the line $s$ are exactly the vertical steps $u_k$ with $a_k = a_i$. The final contact corresponds either to the end of the path $B$ or to a horizontal step: it does not correspond to an area $a_k = a_i$. 
\end{proof}

As for the classical case, we now extend those definitions to intervals of the $m$-Tamari lattice.

\begin{Definition}
\label{def:m-contact-rise-intervals}
Consider an interval $I$ of $\Tamnm$ described by two $m$-ballot paths $B_1$ and $B_2$ with $B_1 \leq B_2$. Then
\begin{enumerate}
\item $\mcontacts(I) = \mcontacts(B_1)$, $\mcontactsStep{i,j}(I):= \mcontactsStep{i,j}(B_1)$ for $1 \leq i \leq n$ and $1 \leq j \leq m$, $\mcontactsV(I):=\mcontactsV(B_1)$, and $\mcontactsP(I,X):=\mcontactsP(B_1,X)$;

\item $\mrisesStep{i}(I):=\mrisesStep{i}(B_2)$ for $0 \leq i \leq mn$, $\mrisesV(I):=\mrisesV(B_2)$, and $\mrisesP(I,Y) := \mrisesP(B_2,Y)$.
\end{enumerate}
To summarize, all the statistics we defined on $m$-ballot paths are extended to $m$-Tamari intervals by looking at the \emph{lower bound} $m$-ballot path $B_1$ when considering contacts and the \emph{upper bound} $m$-ballot path $B_2$ when considering rises. 

Besides, we write $\size(I)$ the size $n$ of the $m$-ballot paths $B_1$ and $B_2$.
\end{Definition}

Finally, the definition of \emph{distance} naturally extends to $m$-Tamari.

\begin{Definition}
\label{def:m-distance}
Let $I = [B_1, B_2]$ be an interval of $\Tamnm$. We call the \emph{distance} of $I$ and write $\distance(I)$ the maximal length of all chains between $B_1$ and $B_2$ in the $m$-Tamari lattice. 
\end{Definition} 

We can now state the generalized version of Theorem~\ref{thm:main-result-classical}.

\begin{Theorem}[general case]
\label{thm:main-result-general}
Let $x,y,t,q$ be variables and $X = (x_0, x_1, x_2, \dots)$ and $Y = (y_0, y_1, y_2, \dots)$ be commutative alphabets. Consider the generating function

\begin{equation}
\Phi_m(t; x, y, X, Y, q) = \sum_{I} t^{\size(I)} x^{\mcontacts(I)} y^{\mrises(I)} \mcontactsP(I,X) \mrisesP(I,Y) q^{\distance(I)}
\end{equation}  

summed over all intervals of the $m$-Tamari lattices. Then, for all $m$, we have

\begin{equation}
\Phi_m(t; x, y, X, Y, q) = \Phi_m(t; y, x, Y, X, q).
\end{equation}
\end{Theorem}

We will give a combinatorial proof of this result, describing an involution on intervals of $m$-Tamari lattices which uses the classical case $\risecontact$ involution defined in Theorem~\ref{thm:rise-contact-statistics}. First, we will recall and reinterpret some results of \cite{IntervalPosetsInitial}. In particular, we recall how intervals of the $m$-Tamari lattice can be seen as interval-posets.

\subsection{$m$-Tamari interval-posets}

The $m$-Tamari lattice $\Tamnm$ is trivially isomorphic to an upper ideal of the classical Tamari lattice $\Tamntm$.

\begin{Definition}
\label{m-dyck-paths}
Let $B$ be an $m$-ballot path, we construct the Dyck path $\DD(B)$ by replacing every vertical step of $B$ by $m$ up-steps and every horizontal step of $B$ by a down-step. The set of such images are called the $m$-Dyck paths.
\end{Definition}

\begin{figure}[ht]
\centering
\input{figures/m-dyck-ballot}
\caption{A $2$-ballot path and its corresponding $2$-Dyck path.}
\label{fig:m-dyck}
\end{figure}
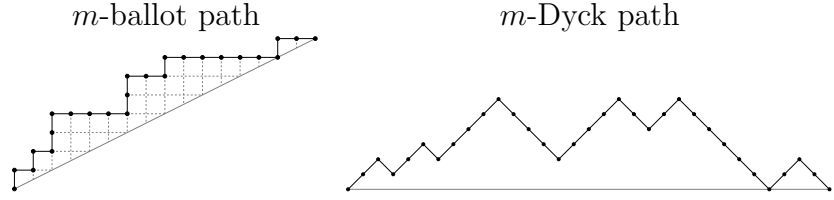

See Figure~\ref{fig:m-dyck} for an example. The $m$-Dyck paths have a trivial characterization: they are the Dyck paths whose rises are divisible by~$m$. In other words, a Dyck path $D$ is an $m$-Dyck path if and only if all values of $\risesV(D)$ are divisible by $m$. We say that they are \emph{rise-$m$-divisible}: the set of $m$-Dyck paths is exactly the set of rise-$m$-divisible Dyck paths. Besides, the set of $m$-Dyck paths is stable by the Tamari rotation. More precisely, they correspond to the upper ideal generated by the Dyck path $(1^m0^m)^n$ which is the image of the initial $m$-ballot path of $\Tamnm$, see Figure~\ref{fig:minimal-m-tam} for an example and \cite{mTamari} for more details.

\begin{figure}[ht]

  \input{figures/minimal-m-tam}

  \caption{Minimal element of $\Tam{3}{2}$.}

  \label{fig:minimal-m-tam}
\end{figure}

We can read the $m$-statistics of an $m$-ballot path on its corresponding $m$-Dyck path.

\begin{Proposition}
\label{prop:m-ballot-m-dyck}
Let $B$ be an $m$-ballot path of size $n$ and $D = \DD(B)$ then
\begin{align}
\mrisesStep{i}(B) &= \frac{1}{m}\risesStep{i}(D) & \text{ for } 0 \leq i \leq nm; \\
\mcontacts(B) &= \contacts(D); \\
\mcontactsStep{i,m}(B) &= \contactsStep{im}(D) & \text{ for } 1 \leq i \leq n; \\
\mcontactsStep{i,j}(B) &= \contactsStep{(i-1)m+j}(D) - 1 & \text{ for } 1 \leq i \leq n \text{ and } 1 \leq j < m.
\end{align}
\end{Proposition}

\begin{proof}
The result is clear for rises. For contacts, note that the $m$-Dyck path can be obtained from the ballot path by sending every point $(x,y)$ of the ballot path to $(my + x, my - x)$. In particular, every contact point between the ballot path and a line of slope $\frac{1}{m}$ is sent to a contact point between the $m$-Dyck path and a horizontal line. When $j\neq m$, the line $\ell_{i,j}$ starts at a non-integer point $(a, b+\frac{j}{m})$ which becomes $(mb + j + a, mb + b -a)$ in the $m$-Dyck path: it now counts for one extra contact when computing $\contactsStep{(i-1)m + j}$ in the $m$-Dyck path. 
\end{proof}

For example, look at Figure~\ref{fig:m-dyck} and its $m$-contact vector on Figure~\ref{fig:m-stat}. The contact vector of its corresponding $2$-Dyck path is given by $\contactsV(D) = (2, \red{2}, 0, \red{3}, 0, \red{1}, 1, \red{1}, 0, \red{1}, 2, \red{1}, 0, \red{1}, 0, \red{1})$: for each even position, the number is the same and for each odd position (in red) the number is increased by 1. The rise-vector of the $m$-Dyck path is $\risesV(D) = (2, 2, 4, 0,0, 0, 4, 0, 2, 0,0,0,0,0, 2, 0)$: it is indeed the $m$-rise-vector of Figure~\ref{fig:m-stat} multiplied by 2.

As the $m$-Tamari lattice can be understood as an upper ideal of the Tamari lattice, it follows that the intervals of $\Tamnm$ are actually a certain subset of intervals of $\Tamntm$: they are the intervals whose both upper and lower bounds are $m$-Dyck paths (in practice, it is sufficient to check that the lower bound is an $m$-Dyck path). It is then possible to represent them as interval-posets. This was done in \cite{IntervalPosetsInitial} where the following characterization was given.

\begin{Definition}
\label{def:m-interval-posets}
An $m$-interval-poset is an interval-poset of size $n\times m$ with
\begin{align}
\label{eq:m-condition}
i m \trprec i m-1 \trprec \dots \trprec i m-(m-1)
\end{align}
for all $1 \leq i \leq n$.
\end{Definition}

\begin{Theorem}[Theorem 4.6 of \cite{IntervalPosetsInitial}]
\label{thm:m-interval-posets}
The $m$-interval-posets of size $n \times m$ are in bijection with intervals of $\Tam{n}{m}$.
\end{Theorem}

On Figure~\ref{fig:m-big-example}, you can see two examples of $m$-interval-posets with $m = 2$ and their corresponding $m$-ballot paths. To construct the interval-posets, you convert the ballot paths into $m$-Dyck paths and use the classical constructions of Propositions~\ref{prop:dyck-dec-forest} and \ref{prop:dyck-inc-forest}. You can check that the result agrees with Definition~\ref{def:m-interval-posets}: for all $k$, $2k \trprec 2k-1$. The proof that it is a bijection uses the notion of $m$-binary trees. These are the binary trees of size $nm$ which belong to the upper ideal of $\Tamntm$ corresponding to the $m$-Tamari lattice. This ideal is generated by the binary tree image of the initial $m$-Dyck path through the bijection of Definition~\ref{def:dyck-tree} as shown in Figure~\ref{fig:minimal-m-tam}. The $m$-binary trees have a $(m+1)$-ary recursive structure: this is the key element to prove Theorem~\ref{thm:m-interval-posets} and we will also use it in this paper. 

\begin{Definition}
\label{def:m-binary-tree}
The $m$-binary trees are defined recursively by being either the empty binary tree or a binary tree $T$ of size $m \times n$ constructed from $m+1$ subtrees $T_L, T_{R_1}, \dots, T_{R_m}$ such that
\begin{itemize}
\item the sum of the sizes of $T_L, T_{R_1}, \dots, T_{R_m}$ is $mn - m$;
\item each subtree $T_L, T_{R_1}, \dots, T_{R_m}$ is itself an $m$-binary tree;
\item and $T$ follows the structure bellow.
\end{itemize}

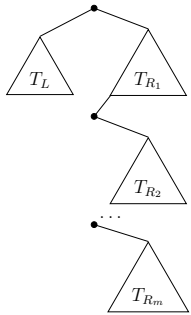
\begin{center}
\scalebox{0.6}{
\input{figures/m-binary-trees}
}
\end{center}

The left subtree of $T$ is $T_L$. The right subtree of $T$ is constructed from $ T_{R_1}, \dots, T_{R_m}$ by the following process: graft a an extra node to the left of the leftmost node of $T_{R_1}$, then graft $T_{R_2}$ to the right of this node, then graft an extra node to the left of the leftmost node of $T_{R_2}$, then graft $T_{R_3}$ to the right of this node, and so on.

Note that in total, $m$ extra nodes were added: we call them the $m$-roots of $T$.
\end{Definition}

Figure~\ref{fig:mbinary} gives two examples of $m$-binary trees for $m=2$ with their decompositions into 3 subtrees. More examples and details about the structure can be found in \cite{IntervalPosetsInitial}. In particular, $m$-binary trees are the images of $m$-Dyck paths through the bijection of Definition~\ref{def:dyck-tree}.

\begin{figure}[ht]
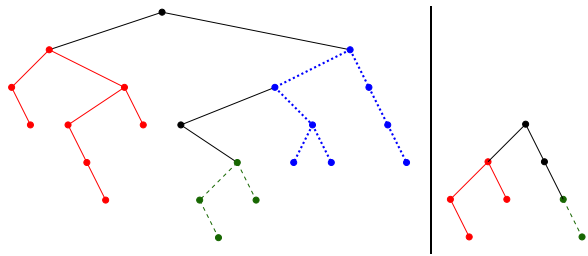

\centering
\begin{tabular}{c|c}
\scalebox{0.5}{\input{figures/mbinary-example}}
&
\scalebox{0.5}{\input{figures/mbinary-example-2}}
\end{tabular}
\caption{Examples of $m$-binary trees for $m=2$: $T_L$ is in red, $T_{R_1}$ is in dotted blue and $T_{R_2}$ is in dashed green. In the second example, $T_{R_1}$ is empty.}
\label{fig:mbinary}
\end{figure}

When working on the classical case, we could safely identify an interval of the Tamari lattice and its representing interval-poset. For $m\neq 1$, we need to be a bit more careful and clearly separate the two notions. Indeed, the $m$-statistics from Definition~\ref{def:m-contact-rise-intervals} of an interval of $\Tamnm$ are not equal to the statistics of its corresponding interval-poset from Definition~\ref{def:contact-rise-intervals}. They can anyway be retrieved through simple operations.

\begin{Proposition}
\label{prop:m-statistics-interval-posets}
Let $I$ be an interval of $\Tamnm$, and $\tI$ its corresponding interval-poset of size $nm$. Then
\begin{align}
\mrisesStep{i}(I) &= \frac{1}{m}\risesStep{i}(\tI) & \text{ for } 0 \leq i \leq nm; \\
\mcontacts(I) &= \contacts(\tI); \\
\mcontactsStep{i,m}(I) &= \contactsStep{im}(\tI) & \text{ for } 1 \leq i \leq n; \\
\mcontactsStep{i,j}(I) &= \contactsStep{(i-1)m+j}(\tI) - 1 & \text{ for } 1 \leq i \leq n \text{ and } 1 \leq j < m; \\
\label{eq:m-distance}
\distance(I) &= \distance(\tI).
\end{align}
\end{Proposition}

\begin{proof}
All identities related to rises and contacts are a direct consequence of Proposition~\ref{prop:m-ballot-m-dyck}. Only~\eqref{eq:m-distance} needs to be proved, which is actually also direct: $\Tamnm$ is isomorphic to the ideal of $m$-Dyck path in $\Tamntm$ and so the distance between two paths in the lattice stays the same. 
\end{proof}

\subsection{The expand-contract operation on $m$-Tamari intervals}

\begin{Definition}
\label{def:m-divisible}
We say that an interval-poset $I$ of size $nm$ is 
\begin{itemize}
\item \emph{contact-$m$-divisible} if all values of $\contactsV(I)$ are divisible by $m$;
\item \emph{rise-$m$-divisible} if all values of $\risesV(I)$ are divisible by $m$;
\item \emph{rise-contact-$m$-divisible} if it is both contact-$m$-divisible and rise-$m$-divisible.
\end{itemize}
\end{Definition}

In particular, $m$-interval-posets are rise-$m$-divisible but not necessary contact-$m$-divisible. Besides, we saw that rise-$m$-divisible Dyck paths were exactly $m$-Dyck paths, but the set of rise-$m$-divisible interval-posets is not equal to $m$-interval-posets. Indeed, an interval whose upper bound is an $m$-Dyck path is rise-$m$-divisible but it can have a lower bound which is not an $m$-Dyck path and so it is not an $m$-interval-poset. 

Furthermore, it is quite clear that the set of $m$-interval-posets is not stable through the rise-contact involution $\risecontact$. Indeed, the image of an $m$-interval-poset would be contact-$m$-divisible but not necessary rise-$m$-divisible. In this section, we describe a bijection between the set of $m$-interval-posets and the set of rise-contact-$m$-divisible intervals. This bijection will allow us to define an involution on $m$-interval-posets which proves Theorem~\ref{thm:main-result-general}.

\begin{Definition}
Let $(T, \ell)$ be a grafting tree of size $nm$ and $v_1, \dots, v_{nm}$ be the nodes of $T$ taken in in-order. We say that $(T, \ell)$ is an $m$-grafting-tree if $\ell(v_i) \geq 1$ for all $i$ such that $i \not \equiv 0 \mod m$.
\end{Definition}

\begin{Proposition}
\label{prop:m-graft-m-interval}
An interval-poset $I$ is an $m$-interval-poset if and only if $\graftingTree(I)$ is an $m$-grafting-tree.
\end{Proposition}

As an example, the top and bottom grafting trees of Figure~\ref{fig:m-big-example} are $m$-grafting trees: you can check that every odd node has a non-zero label. The corresponding $m$-interval-posets are drawn on the same lines. Proposition~\ref{prop:m-graft-m-interval} is direct consequence of 
Proposition~\ref{prop:grafting-direct} and Definition~\ref{def:m-interval-posets}. Indeed, to obtain \eqref{eq:m-condition}, it sufficient to say that every node $i$ of the interval-poset such that  $i \not \equiv 0 \mod m$ has at least one decreasing child $j > i$ such that $j \trprec i$. By definition of an interval-poset, this gives $i + 1 \trprec i$. 

\begin{Proposition}
\label{prop:m-graft-m-binary}
Let $(T,\ell)$ be an $m$-grafting-tree, then $T$ is an $m$-binary-tree.
\end{Proposition}

\begin{proof}
This is immediate by Proposition~\ref{prop:grafting-direct}: $(T,\ell)$ corresponds to an $m$-interval-poset $I$. In particular, the upper bound of $I$ is an $m$-binary tree which is equal to $T$.
\end{proof}

\begin{Proposition}
\label{prop:m-exp-cont}
Let $(T, \ell)$ be an $m$-grafting-tree, and $v_1, \dots, v_{nm}$ its nodes taken in in-order. The \emph{expansion of $(T,\ell)$} is $\expand(T,\ell) = (T', \ell')$ defined by
\begin{itemize}
\item $T' = T$;
\item $\ell'(v_i) = m \ell(v_i)$ if $i \equiv 0 \mod m$, otherwise, $\ell'(v_i) = m (\ell(v_i) -~1)$.
\end{itemize}

Then $\expand$ defines a bijection through their grafting trees between $m$-interval-posets and rise-contact-$m$-divisible interval posets. The reverse operation is called \emph{contraction}, we write $(T,\ell) = \contract(T, \ell')$. Besides, we have
\begin{equation}
\label{eq:m-exp-contacts}
\contacts(T,\ell') = m \contacts(T,\ell). 
\end{equation}
Note that we write $\contacts(T,\ell)$ for $\contacts(\graftingTree^{-1}(T,\ell))$ for short.
\end{Proposition}

The intuition behind this operation is first that the relations $im \trprec \dots \trprec im - (m-1)$ are not necessary to recover the $m$-interval-poset (because they are always present) and secondly that the structure of the $m$-binary tree allows to replace each remaining decreasing relations by $m$ decreasing relations. Nevertheless, even if the operation is easy to follow on grafting tree (and the proof mostly straight forward), we would very much like to see a ``better'' description of it directly on Tamari intervals.

\begin{proof}
This proposition contains different results, which we organize as claims and prove separately.
\begin{proofclaim}
\label{claim:expand}
$(T, \ell') = \expand(T, \ell)$ is a grafting tree such that $\contacts(T, \ell') = m\contacts(T, \ell)$. 
\end{proofclaim}
These two properties are intrinsically linked, we will prove both at the same time by induction on the recursive structure of $m$-binary-trees. Let $(T, \ell)$ be an $m$-grafting tree. By Proposition~\ref{prop:m-graft-m-binary}, $T$ is an $m$-binary tree. If $T$ is empty, then there is nothing to prove. Let us suppose that $T$ is non-empty: it can be decomposed into $m+1$ subtrees $T_L, T_{R_1}, \dots, T_{R_m}$ which are all $m$-grafting trees. By induction, we suppose that they satisfy the claim.

Let us first focus on the case where $T_L$ is the empty tree. Then $v_1$ (the first node in in-order) is the root, and moreover, the $m$-roots are $v_1, \dots v_m$. We call $T_1, T_2, \dots, T_m$ the subtrees of $T$ whose roots are respectively $v_1, \dots, v_m$ (in particular, $T_1 = T$). See Figure~\ref{fig:m-exp-cont} for an illustration.

\begin{figure}[ht]
\input{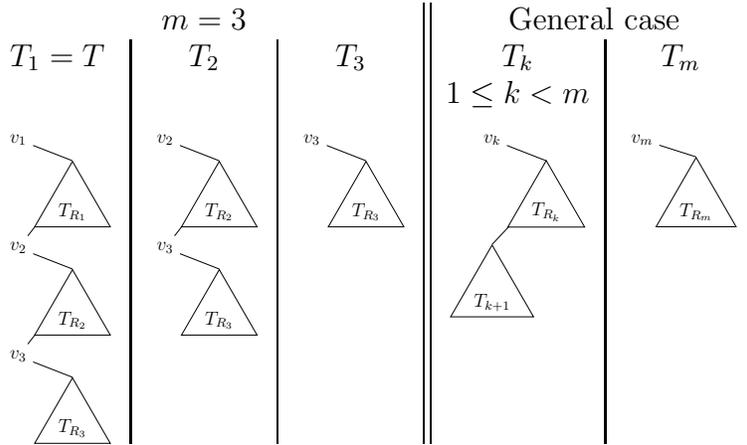}
\caption{Illustration of $T_1, \dots, T_m$}
\label{fig:m-exp-cont}
\end{figure}

In particular, for $ 1 \leq k < m$, the tree $T_k$ follows a structure that depends on $T_{R_k}$ and $T_{k+1}$ as shown in Figure~\ref{fig:m-exp-cont} and $T_m$ depends only on $T_{R_m}$. Note that $T_2, \dots T_k$ are grafting trees but they are not $m$-grafting trees whereas $T_{R_1}, \dots, T_{R_m}$ are. Following Definition~\ref{def:grafting-tree}, the structure gives us
\begin{align}
\label{eq:m-exp-cont-tvk}
\ell(v_k) &\leq \size(T_{R_k}) + \size(T_{k+1}) - \labels(T_{R_k}, \ell) - \labels(T_{k+1}, \ell) \\
\nonumber
 &= \contacts(T_{R_k}, \ell) + \contacts(T_{k+1},\ell)
\end{align}
for $1 \leq k < m$ and
\begin{equation}
\label{eq:m-exp-cont-tvm}
\ell(v_m) \leq \contacts(T_{R_m}, \ell).
\end{equation}

Also, for $1 \leq k < m$, we have $\ell'(v_k) = m(\ell(v_k) -1) \geq 0$ (indeed remember that $\ell(v_k) \geq 1$ because $(T,\ell)$ is an $m$-grafting-tree) and $\ell'(v_m) = m \ell(v_m) \geq 0$. To prove that $(T,\ell')$ is a grafting tree, we need to show
\begin{align}
\label{eq:m-exp-cont-graft}
 \ell'(v_k) &\leq \contacts(T_{R_k}, \ell') + \contacts(T_{k+1},\ell'); \\
\label{eq:m-exp-cont-graftm}
 \ell'(v_m) &\leq \contacts(T_{R_m}, \ell').
\end{align}
We simultaneously prove 
\begin{equation}
\label{eq:m-exp-cont-contacts}
\contacts(T_k, \ell') = m \contacts(T_k, \ell) - k + 1.
\end{equation}
The case $k=1$ in \eqref{eq:m-exp-cont-contacts} finishes to prove the claim.

We start with $k=m$ and then do an induction on $k$ decreasing down to 1. By hypothesis, we know that $(T_{R_m}, \ell)$ satisfies the claim. In particular $(T_{R_m}, \ell')$ is a grafting tree and $\contacts(T_{R_m}, \ell') = m \contacts(T_{R_m}, \ell)$. By definition, we have $\ell'(v_m) = m \ell(v_m)$ and so \eqref{eq:m-exp-cont-tvm} implies \eqref{eq:m-exp-cont-graftm}. Besides
\begin{align}
\contacts(T_m, \ell) &= \size(T_m) - \labels(T_m, \ell) \\
\nonumber
&= 1 + \size(T_{R_m}) - \ell(v_m) - \labels(T_{R_m}, \ell) \\
\nonumber
&= 1 - \ell(v_m) + \contacts(T_{R_m}, \ell), \\
\label{eq:m-exp-cont-contacts-m}
m \contacts(T_m, \ell) - m + 1 &= m -  m\ell(v_m) + m \contacts(T_{R_m}, \ell) - m + 1 \\
\nonumber
&= 1 - \ell'(v_m) + \contacts(T_{R_m}, \ell') \\
\nonumber
&= \contacts(T_m, \ell'),
\end{align}
\emph{i.e.}, case $k=m$ of \eqref{eq:m-exp-cont-contacts}.

Now, we choose $1 \leq i <m$ and assume \eqref{eq:m-exp-cont-graft} and \eqref{eq:m-exp-cont-contacts} to be true for $k > i$. We have $\ell'(v_i) = m \left( \ell(v_i) - 1 \right)$, so \eqref{eq:m-exp-cont-tvk} gives us
\begin{align}
\ell'(v_i) &\leq m \contacts(T_{R_i}, \ell) + m \contacts(T_{i+1}, \ell) - m \\
\nonumber
&= \contacts(T_{R_i}, \ell') + \contacts(T_{i+1}, \ell') + i - m
\end{align}
using \eqref{eq:m-exp-cont-contacts} with $k = i+1$. As $i < m$, this proves \eqref{eq:m-exp-cont-graft} for $k=i$. Now, the structure of $T_i$ gives us
\begin{equation}
\contacts(T_i, \ell) = \contacts(T_{R_i}, \ell) + \contacts(T_{i+1}, \ell) + 1 - \ell(v_i);
\end{equation}
\begin{align}
\label{eq:m-exp-cont-contactsi}
\contacts(T_i, \ell') &= \contacts(T_{R_i}, \ell') + \contacts(T_{i+1}, \ell') + 1 - \ell'(v_i) \\
\nonumber
&= m \contacts(T_{R_i}, \ell) + m \contacts(T_{i+1}, \ell) - (i+1) + 1 + 1 - m (\ell(v_i) - 1) \\
\nonumber
&= m \contacts(T_i, \ell) - i + 1.
\end{align}

The case where $T_L$ is not the empty tree is left to consider but actually follows directly. The claim is true on $T_L$ by induction as its size is strictly smaller than $T$. Let $\tilde{T}$ be the tree $T$ where you remove the left subtree $T_L$. Then $\tilde{T}$ is still an $m$-grafting tree and the above proof applies. The expansion on $T$ consists of applying the expansion independently on $T_L$ and $\tilde{T}$ and we get $\contacts(T, \ell') = \contacts(T_L, \ell') + \contacts(\tilde{T}, \ell') = m \contacts(T, \ell)$.

\begin{proofclaim}
$(T, \ell') = \expand(T)$ is rise-contact-$m$-divisible.
\end{proofclaim}
$T$ is still an $m$-binary tree, which by Proposition~\ref{prop:grafting-direct}, means that the upper bound of $\graftingTree^{-1}(T, \ell')$ is an $m$-binary tree: it corresponds to an $m$-Dyck path and is then $m$-rise-divisible. We have just proved that $\contacts(T, \ell') = m \contacts(T, \ell)$ is a multiple of $m$. By Proposition~\ref{prop:grafting-contact} the rest of the contact vector is given by reading the labels on $T$: by definition of $\ell'$, all labels are multiples of $m$.

\begin{proofclaim}
Let $(T, \ell')$ be a rise-contact-$m$-divisible grafting tree, then $(T, \ell) = \contract(T, \ell')$ is an $m$-grafting tree.
\end{proofclaim}

We define $(T, \ell) = \contract(T, \ell')$ to make it the inverse of the $\expand$ operation:
\begin{align}
\ell(v_i) &= \frac{\ell'(v_i)}{m} &\text{if } i \equiv 0 \mod m \\
\ell(v_{i}) &= \frac{\ell'(v_i)}{m} + 1 &\text{otherwise}.
\end{align}

As earlier, we simultaneously prove that $(T, \ell)$ is an $m$-grafting tree and that $\contacts(T,\ell) = \frac{\contacts(T,\ell')}{m}$. Our proof follows the exact same scheme as for Claim~\ref{claim:expand}. First note that the fact that $(T,\ell')$ is rise-$m$-divisible implies that $T$ is an $m$-binary tree: indeed, it corresponds to a certain Dyck path which is rise-$m$-divisible. When $T$ is not empty, we can recursively decompose it into $T_L$, $T_{R_1}, \dots, T_{R_m}$. As earlier, the only case to consider is actually when $T_L$ is empty. We use the decomposition of $T$ depicted in Figure~\ref{fig:m-exp-cont} and prove \eqref{eq:m-exp-cont-contacts} and~\eqref{eq:m-exp-cont-tvk} by induction on $k$ decreasing from $m$ to $1$. The case where $k = m$ is straightforward: we have that \eqref{eq:m-exp-cont-graftm} implies \eqref{eq:m-exp-cont-tvm} and \eqref{eq:m-exp-cont-contacts-m} is still true. Now, we choose $1 \leq i <m$ and assume \eqref{eq:m-exp-cont-tvk} and \eqref{eq:m-exp-cont-contacts} to be true for $k > i$. Using \eqref{eq:m-exp-cont-graft}, we get
\begin{align}
m ( \ell(v_i) - 1) &\leq \contacts(T_{R_i}, \ell') + \contacts(T_{i+1}, \ell') \\
\nonumber 
&= m \contacts(T_{R_i}, \ell) + m \contacts(T_{i+1}, \ell) -(i+1) +1 \\
\nonumber
\ell(v_i) &\leq \contacts(T_{R_i}, \ell) + \contacts(T_{i+1}, \ell) - \frac{i}{m} + 1 .
\end{align}
We have $0 < \frac{i}{m} < 1$ and because $\ell(v_i)$ is an integer then \eqref{eq:m-exp-cont-tvk} is true. Besides, by definition of $\ell$, $\ell(v_i) \geq 1$, which satisfies the $m$-grafting tree condition. The rest of the induction goes smoothly because \eqref{eq:m-exp-cont-contactsi} is still valid.
\end{proof}

The $\expand$ and $\contract$ operations are the final crucial steps that allow us to define the $m$-contact-rise involution and prove Theorem~\ref{thm:main-result-general}. Before that, we need a last property to understand how the distance statistic behaves through the transformation.

\begin{Proposition}
\label{prop:m-exp-distance}
Let $(T, \ell)$ be an $m$-grafting tree of size $mn$, and $(T,\ell') = \expand(T,\ell)$, then 
\begin{equation}
\distance(T,\ell') = m \distance(T, \ell) + \frac{n m (m-1)}{2}
\end{equation}
\end{Proposition}

\begin{proof}
For each vertex $v_i$ of $T$, let $d_i(T,\ell) = \size(T_R(v_i)) - \labels(T_R(v_i), \ell) - \ell(v_i)$ where $T_R(v_i)$ is the right subtree of the vertex $v_i$ in $T$ and remember that $d(T,\ell) = \sum_{i=1}^{nm} d_i$ by Proposition~\ref{prop:grafting-tree-distance}. We claim that
\begin{equation}
d_{im - j}(T, \ell') = m d_{im - j}(T, \ell) + j  
\end{equation}
for $1 \leq i \leq n$ and $0 \leq j < m$, which gives the result by summation. We prove our claim by induction on $n$. Let us suppose that $T$ is not empty and decomposes into $T_L, T_{R_1}, \dots, T_{R_m}$. The result is true by induction on the subtrees: indeed the index of a given vertex (in in-order) in $T$ and in its corresponding subtree is the same modulo $m$. It remains to prove the property for the $m$-roots of $T$, which are given by $\lbrace v_{im - j} ; 0 \leq j < m \rbrace$ for some $1 \leq i \leq n$. We use the decomposition of Figure~\ref{fig:m-exp-cont}. Remember that $T_{R_m}$ is an $m$-grafting tree and we have by Proposition~\ref{prop:m-exp-cont} that $\contacts(T_{R_m},\ell') = m\contacts(T_{R_m},\ell)$. We get
\begin{align}
d_{im}(T, \ell') &= \size(T_{R_m}) - \labels(T_{R_m}, \ell') - \ell'(v_{im}) \\
\nonumber
&= \contacts(T_{R_m},\ell') - \ell'(v_{im}) \\
\nonumber
&= m \contacts(T_{R_m},\ell) - m \ell(v_{im}) \\
\nonumber
&= m d_{im}(T,\ell).
\end{align}

Now, remember that by the decomposition of Figure~\ref{fig:m-exp-cont}, $T_R(v_{im-j})$ is made of $T_{R_{m-j}}$ (which is an $m$-grafting tree) with $T_{m-j+1}$ grafted on its left most branch, using that and \eqref{eq:m-exp-cont-contacts}, we get
\begin{align}
d_{im-j}(T,\ell') &= \size(T_R(v_{im-j})) - \labels(T_R(v_{im-j}),\ell') - \ell'(v_{im-j}) \\
\nonumber
&= \size(T_{R_{m-j}}) + \size(T_{m-j+1}) - \labels(T_{R_{m-j}}, \ell') - \labels(T_{m-j+1},\ell') - \ell'(v_{im-j}) \\
\nonumber
&= \contacts(T_{R_{m-j}}, \ell') + \contacts(T_{m-j+1},\ell') -  \ell'(v_{im-j}) \\
\nonumber
&= m \contacts(T_{R_{m-j}}, \ell) + m \contacts(T_{m-j+1},\ell) -(m-j+1) + 1 - m(\ell(v_{im-j}) -1) \\
\nonumber
&= d_{im-j}(T,\ell) +j.
\end{align}
\end{proof}

\begin{Theorem}[The $m$-rise-contact involution]
\label{thm:m-rise-contact-involution}
Let $\mrisecontact$ be the \emph{$m$-rise-contact involution} defined on $m$-interval-posets by
\begin{align}
\mrisecontact = \contract \circ \risecontact \circ \expand
\end{align}
Then $\mrisecontact$ is an involution on intervals of $\Tamnm$, such that for an interval~$I$ and a commutative alphabet $X$,
\begin{align}
\label{eq:m-rise-contact-rises}
\mrises(I) &= \mcontacts(\mrisecontact(I)); \\
\label{eq:m-rise-contact-partitions}
\mrisesP(I,X) &= \mcontactsP(\mrisecontact(I),X); \\
\label{eq:m-rise-contact-distance}
\distance(I) &= \distance(\mrisecontact(I)).
\end{align}
\end{Theorem}

\begin{proof}
Le $I$ be an interval of $\Tamnm$ with $\tilde{I}$ its corresponding $m$-interval-poset in $\Tamntm$ and let $(T,\ell) = \expand(\tilde{I})$ be the expansion of its $m$-grafting-tree. By Propositions~\ref{prop:m-statistics-interval-posets} and~\ref{prop:m-exp-cont}, we have
\begin{align}
\label{eq:m-rise-contact-proof1}
\contacts(T,\ell) &= m \contacts(\tilde{I}) = m (\mcontacts(I)) \\
\label{eq:m-rise-contact-proof2}
\contactsStep{im}(T, \ell) &= m \contactsStep{im}(\tilde{I}) = m (\mcontactsStep{i,m}(I)) \\
\label{eq:m-rise-contact-proof3}
\contactsStep{(i-1)m + j}(T, \ell) &= m (\contactsStep{(i-1)m +j}(\tilde{I}) - 1) = m (\mcontactsStep{i,j}(I))
\end{align}
for $1 \leq i \leq n$ and $1 \leq j < m$. And using again Propositions~\ref{prop:m-statistics-interval-posets} and the fact that the expansion does not affect the initial forest, we have
\begin{align}
\label{eq:m-rise-contact-proof4}
\risesStep{i}(T,\ell) &= \risesStep{i}(\tilde{I}) = m (\mrisesStep{i}(I))
\end{align}
for $0 \leq i \leq mn$. In other words, $(T,\ell)$ is rise-contact-$m$-divisible.  Let $(T', \ell') = \risecontact(T, \ell)$. By Theorem~\ref{thm:rise-contact-statistics}, we have that

\begin{align}
\label{eq:m-rise-contact-proof5}
\rises(T,\ell) &= \contacts(T',\ell'); \\
\label{eq:m-rise-contact-proof6}
\risesP(T,\ell,X) &= \contactsP(T',\ell',X); \\
\label{eq:m-rise-contact-proof7}
\distance(T,\ell) &= \distance(T',\ell').
\end{align}

In particular, this means that $(T',\ell')$ is still rise-contact-$m$-divisible: we can apply the $\contract$ operation and we get an $m$-interval-poset  $\tilde{J}$ of $\Tamntm$, which corresponds to some interval $J$ of $\Tamnm$. This proves that $\mrisecontact$ is well defined and is an involution by construction. Using \eqref{eq:m-rise-contact-proof4} followed by \eqref{eq:m-rise-contact-proof5} then by \eqref{eq:m-rise-contact-proof1} on $(T',\ell')$, $J$ and $\tilde{J}$, we obtain \eqref{eq:m-rise-contact-rises}. 

The result \eqref{eq:m-rise-contact-partitions} follows in a similar way. The equality \eqref{eq:m-rise-contact-proof4} tells us that the rise vector of $(T,\ell)$ is the rise vector of $I$ where every value has been multiplied by $m$. Now \eqref{eq:m-rise-contact-proof6} basically says that the contact vector of $(T',\ell')$ is a permutation of the rise vector of $(T,\ell)$. Finally, we apply \eqref{eq:m-rise-contact-proof1}, \eqref{eq:m-rise-contact-proof2}, and \eqref{eq:m-rise-contact-proof3} on $(T',\ell')$, $\tilde{J}$ and $J$ instead of $(T,\ell)$, $\tilde{I}$ and $I$, and we get the equality \eqref{eq:m-rise-contact-partitions} between the rise and contact partitions.

For \eqref{eq:m-rise-contact-distance}, see in \eqref{eq:m-rise-contact-proof7} that the distance statistic is not affected by $\risecontact$. Proposition~\ref{prop:m-exp-distance} tells us that $\expand$ applies an affine transformation which does not depend on the shape of $T$, it is then reverted by the application of $\contract$ later on.
\end{proof}

Figure~\ref{fig:m-big-example} shows a complete example of the $\mrisecontact$ involution on an interval of $\Tam{11}{2}$. You can run more examples and compute tables for all intervals using the provided live {\tt Sage-Jupyter} notebook~\cite{JNotebook}.

\bibliographystyle{alpha}
\bibliography{paper}

\begin{landscape}
\begin{figure}[p]
\input{figures/m_big_example}
\caption{The $m$-rise-contact involution on an example}
\label{fig:m-big-example}
\end{figure}
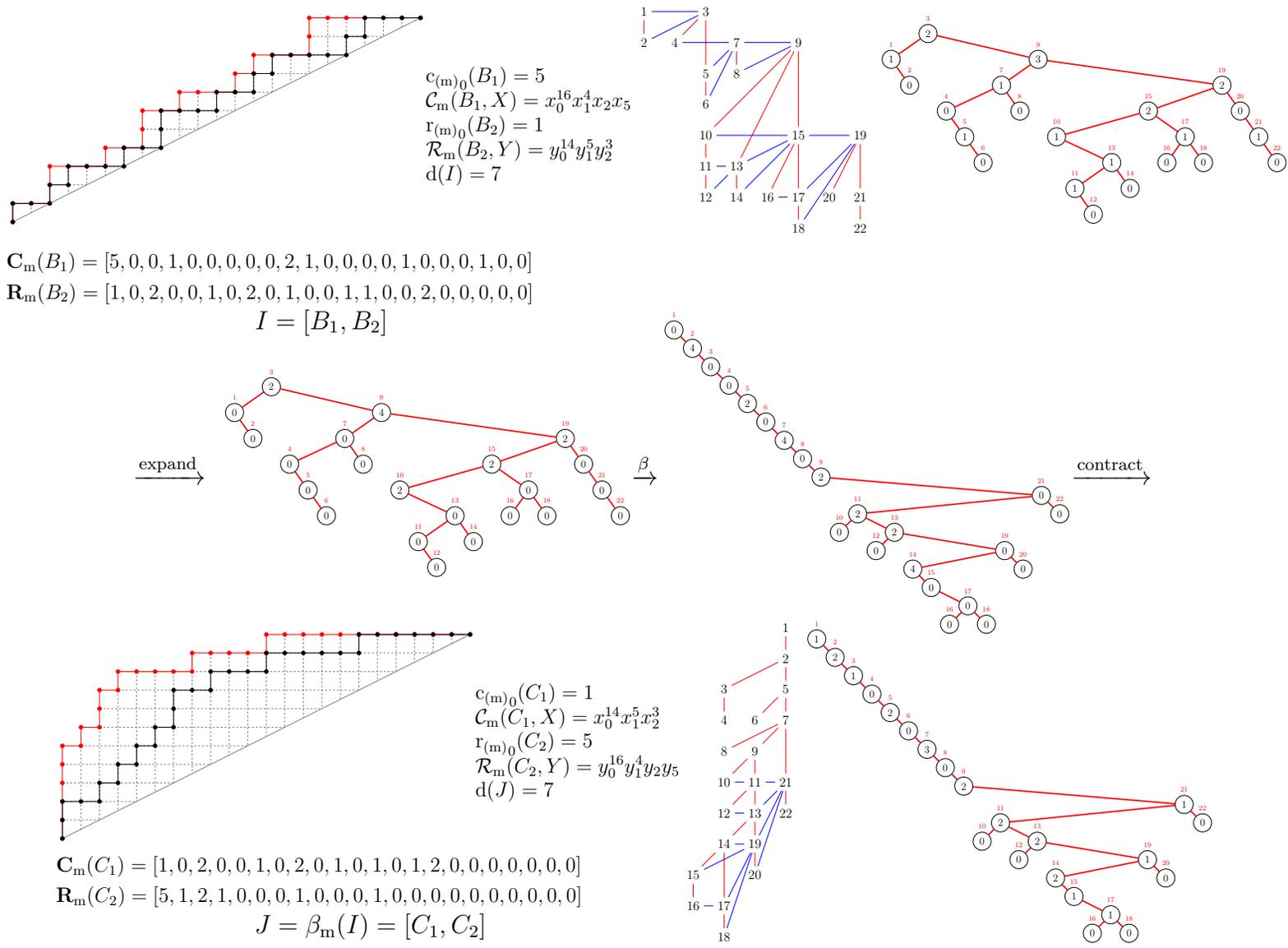
\end{landscape}

\end{document}

%% file: figures/trees/T3-2.tex


\begin{tikzpicture}
\node (N0) at (1.250, 0.000){};
\node (N00) at (0.250, -0.500){};
\node (N001) at (0.750, -1.000){};
\draw[Point] (N001) circle;
\draw (N00.center) -- (N001.center);
\draw[Point] (N00) circle;
\draw (N0.center) -- (N00.center);
\draw[Point] (N0) circle;
\end{tikzpicture}

%% file: figures/involution_dyckpaths.tex
\begin{tabular}{cccc}
\scalebox{.3}{\input{figures/dyck/D10-ex1}}
&
$\xrightarrow{\text{reverse}}$
&
\scalebox{.3}{\input{figures/dyck/D10-ex2}}
&
\scalebox{.3}{\input{figures/trees/T10-ex1}}
\\
\begin{tabular}{l}
$\contactsP = x_0^4 x_1^2 x_2^4$  \\
 $\risesP = y_0^5 y_2^5$
\end{tabular}
& &
$\left\downarrow \text{Tamari symmetry} \right.$
&
\\
\scalebox{.3}{\input{figures/dyck/D10-ex4}}
&
$\xleftarrow{\text{reverse}}$
&
\scalebox{.3}{\input{figures/dyck/D10-ex3}}
&
\scalebox{.3}{\input{figures/trees/T10-ex2}}
\\
\begin{tabular}{l}
$\contactsP = x_0^5 x_2^5 $  \\
$\risesP = y_0^4 y_1^2 y_2^4$
\end{tabular}
&
&
&
\end{tabular}

%% file: figures/grafting.tex
\begin{tabular}{c|c}
$\begin{aligned}\scalebox{0.8}{\input{figures/interval-posets/I3-ex2}}\end{aligned}
\pleft
\begin{aligned}\scalebox{0.8}{\input{figures/interval-posets/I3-ex1}}\end{aligned} =
\begin{aligned}\scalebox{0.8}{\input{figures/interval-posets/I6-ex2}}\end{aligned} $
&
$\begin{aligned}\scalebox{0.8}{\input{figures/interval-posets/I3-ex2}}\end{aligned}
\pright{0}
\begin{aligned}\scalebox{0.8}{\input{figures/interval-posets/I3-ex1}}\end{aligned} =
\begin{aligned}\scalebox{0.8}{\input{figures/interval-posets/I6-ex3}}\end{aligned} $
 \\
 \hline
$\begin{aligned}\scalebox{0.8}{\input{figures/interval-posets/I3-ex2}}\end{aligned}
\pright{1}
\begin{aligned}\scalebox{0.8}{\input{figures/interval-posets/I3-ex1}}\end{aligned} =
\begin{aligned}\scalebox{0.8}{\input{figures/interval-posets/I6-ex4}}\end{aligned} $
&
$\begin{aligned}\scalebox{0.8}{\input{figures/interval-posets/I3-ex2}}\end{aligned}
\pright{2}
\begin{aligned}\scalebox{0.8}{\input{figures/interval-posets/I3-ex1}}\end{aligned} =
\begin{aligned}\scalebox{0.8}{\input{figures/interval-posets/I6-ex5}}\end{aligned}$
\end{tabular}

%% file: figures/interval-posets/I3-ex2.tex

\begin{tikzpicture}[xscale=0.8, yscale=0.8]
\node(T2) at (0,-1) {2};
\node(T1) at (0,0) {1};
\node(T3) at (1,-1) {3};
\draw[line width = 0.5, color=blue] (T2) -- (T3);
\draw[line width = 0.5, color=red] (T2) -- (T1);
\end{tikzpicture}

%% file: figures/interval-posets/I3-ex1.tex

\begin{tikzpicture}[xscale=0.5, yscale=0.8]
\node(T1) at (0,0) {1};
\node(T2) at (1,0) {2};
\node(T3) at (1,-1) {3};
\draw[line width = 0.5, color=red] (T3) -- (T2);
\end{tikzpicture}

%% file: figures/interval-posets/I6-ex3.tex

\begin{tikzpicture}[scale=0.8]
\node(T2) at (0,-1) {2};
\node(T1) at (0,0) {1};
\node(T3) at (1,-1) {3};
\node(T4) at (1,-2) {4};
\node(T5) at (2,-2) {5};
\node(T6) at (2,-3) {6};
\draw[line width = 0.5, color=blue] (T2) -- (T3);
\draw[line width = 0.5, color=red] (T2) -- (T1);
\draw[line width = 0.5, color=red] (T6) -- (T5);
\end{tikzpicture}

%% file: figures/m-contacts-rises-example.tex
\begin{tabular}{cl}
$\begin{aligned}\scalebox{0.7}{\input{figures/m-contacts-rises-example-path}}\end{aligned}$
&
\scalebox{.8}{
\begin{tabular}{l}
$\mcontacts(B) = 2$ \\
$\mcontactsV(B) = \left[2,1,0,2,0,0,1,0,0,0,2,0,0,0,0,0 \right]$ \\
$\mcontactsP(B, X) =  x_0^{11} x_1^2 x_2^3$ \\
$\mrises(B) = 1$ \\
$\mrisesV(B) = \left[1,1,2,0,0,0,2,0,1,0,0,0,0,0,1,0 \right]$ \\
$\mrisesP(B,Y) = y_0^{10} y_1^4 y_2^2$
\end{tabular}
}
\end{tabular}

%% file: figures/m-dyck-ballot.tex

\def \fpath{figures/}

\begin{tabular}{cc}
$m$-ballot path & $m$-Dyck path \\
\scalebox{0.25}{\input{\fpath mpaths/P8-2-ex1}} &
\scalebox{0.2}{\input{\fpath dyck/D16-ex1}}
\end{tabular}

%% file: figures/m_big_example.tex
\setlength\tabcolsep{-1pt}
\begin{tabular}{cccc}
$\begin{aligned}\scalebox{.3}{\input{figures/mpaths/P11-interval}}\end{aligned}$ 
&
\scalebox{.8}{
\begin{tabular}{l}
$\mcontacts(B_1) = 5$ \\
$\mcontactsP(B_1, X) =  x_0^{16} x_1^4 x_2 x_5$ \\
$\mrises(B_2) = 1$ \\
$\mrisesP(B_2,Y) = y_0^{14} y_1^5 y_2^3$ \\
$\distance(I) = 7$
\end{tabular}
}
&
$\begin{aligned}\scalebox{.5}{\input{figures/interval-posets/I22-ex1}}\end{aligned}$ &
$\begin{aligned}\scalebox{.4}{\input{figures/m-grafting-tree}}\end{aligned}$ \\
\multicolumn{2}{l}{\scalebox{.8}{$\mcontactsV(B_1) = \left[5, 0, 0, 1, 0, 0, 0, 0, 0, 2, 1, 0, 0, 0, 0, 1, 0, 0, 0, 1, 0, 0 \right]$}} \\
\multicolumn{2}{l}{\scalebox{.8}{$\mrisesV(B_2) = \left[1, 0, 2, 0, 0, 1, 0, 2, 0, 1, 0, 0, 1, 1, 0, 0, 2, 0, 0, 0, 0, 0\right]$}} \\
\multicolumn{2}{c}{$I = [B_1, B_2]$} 
\end{tabular}

\vspace{-0.5cm}

\begin{tabular}{ccccc}
$\xrightarrow{\expand}$
$\begin{aligned}\scalebox{.4}{\input{figures/grafting-tree-expanded}}\end{aligned}$ & $\xrightarrow{\risecontact}$ & 
$\begin{aligned}\scalebox{.4}{\input{figures/grafting-tree-expanded2}}\end{aligned}$ & $\xrightarrow{\contract}$
\end{tabular}

\vspace{-0.5cm}

\begin{tabular}{ccc}
\begin{tabular}{cc}
$\begin{aligned}\scalebox{.3}{\input{figures/mpaths/P11-interval-2}}\end{aligned}$ 
&
\scalebox{.8}{
\begin{tabular}{l}
$\mcontacts(C_1) = 1$ \\
$\mcontactsP(C_1, X) =  x_0^{14} x_1^5 x_2^3$ \\
$\mrises(C_2) = 5$ \\
$\mrisesP(C_2,Y) = y_0^{16} y_1^4 y_2 y_5$ \\
$\distance(J) = 7$
\end{tabular}
} \\
\multicolumn{2}{l}{\scalebox{.8}{$\mcontactsV(C_1) = \left[1, 0, 2, 0, 0, 1, 0, 2, 0, 1, 0, 1, 0, 1, 2, 0, 0, 0, 0, 0, 0, 0 \right]$}}\\
\multicolumn{2}{l}{\scalebox{.8}{$\mrisesV(C_2) = \left[5, 1, 2, 1, 0, 0, 0, 1, 0, 0, 0, 1, 0, 0, 0, 0, 0, 0, 0, 0, 0, 0\right]$}}\\
\multicolumn{2}{c}{$J = \mrisecontact(I) = [C_1, C_2]$}
\end{tabular}
&
$\begin{aligned}\scalebox{.5}{\input{figures/interval-posets/I22-ex2}}\end{aligned}$ &
$\begin{aligned}\scalebox{.4}{\input{figures/m-grafting-tree2}}\end{aligned}$ \\
\end{tabular}

%% file: paper.bbl
\begin{thebibliography}{BMFPR11}

\bibitem[BB09]{BijTriangulations}
O.~Bernardi and N.~Bonichon.
\newblock {Catalan's intervals and realizers of triangulations}.
\newblock {\em Journal of Combinatorial Theory Series A}, 116(1):55--75, 2009.

\bibitem[BMFPR11]{mTamari}
M.~Bousquet-M{\'e}lou, E.~Fusy, and L.-F. Pr{\'e}ville-Ratelle.
\newblock The number of intervals in the {$m$}-{T}amari lattices.
\newblock {\em Electron. J. Combin.}, 18(2):Paper 31, 26, 2011.

\bibitem[BPR12]{BergmTamari}
F.~Bergeron and L.-F. Pr{\'e}ville-Ratelle.
\newblock Higher trivariate diagonal harmonics via generalized {T}amari posets.
\newblock {\em J. Comb.}, 3(3):317--341, 2012.

\bibitem[CCP14]{ME_FPSAC2014}
F.~{Chapoton}, G.~{Chatel}, and V.~{Pons}.
\newblock Two bijections on {T}amari intervals.
\newblock {\em DMTCS Proceedings, 26th International Conference on Formal Power
  Series and Algebraic Combinatorics}, 2014.

\bibitem[CFLM18]{WeylChamber}
J.~Courtiel, E.~Fusy, M.~Lepoutre, and M.~Mishna.
\newblock Bijections for weyl chamber walks ending on an axis, using arc
  diagrams and schnyder woods.
\newblock {\em European Journal of Combinatorics}, 69:126 -- 142, 2018.

\bibitem[Cha07]{Chap}
F.~Chapoton.
\newblock Sur le nombre d'intervalles dans les treillis de {T}amari.
\newblock {\em S\'em. Lothar. Combin.}, 55:Art. B55f, 18 pp., 2005/07.

\bibitem[CKS09]{DecompBeta10}
A.~Claesson, S.~Kitaev, and E.~Steingrímsson.
\newblock Decompositions and statistics for $\beta$(1,0)-trees and nonseparable
  permutations.
\newblock {\em Advances in Applied Mathematics}, 42(3):313 -- 328, 2009.

\bibitem[CKS13]{InvolutionBeta10}
A.~Claesson, S.~Kitaev, and E.~Steingrímsson.
\newblock An involution on $\beta$(1,0)-trees.
\newblock {\em Advances in Applied Mathematics}, 51(2):276 -- 284, 2013.

\bibitem[CP15]{IntervalPosetsInitial}
G.~Châtel and V.~Pons.
\newblock Counting smaller elements in the {T}amari and m-{T}amari lattices.
\newblock {\em Journal of Combinatorial Theory, Series A}, 134:58 -- 97, 2015.

\bibitem[CS03]{CoriSchaefferDescTrees}
R.~Cori and G.~Schaeffer.
\newblock Description trees and tutte formulas.
\newblock {\em Theoretical Computer Science}, 292(1):165 -- 183, 2003.
\newblock Selected Papers in honor of Jean Berstel.

\bibitem[Fan18]{FangBeta10}
W.~Fang.
\newblock A trinity of duality: non-separable planar maps, $\beta$-(1,0) trees
  and synchronized intervals.
\newblock {\em Advances in Applied Mathematics}, 95:1--, 04 2018.

\bibitem[FPR17]{FangPrevilleRatelle}
W.~Fang and L.-F. Pr{\'e}ville-Ratelle.
\newblock The enumeration of generalized tamari intervals.
\newblock {\em Eur. J. Comb.}, 61(C):69--84, March 2017.

\bibitem[HT72]{Tamari2}
S.~Huang and D.~Tamari.
\newblock Problems of associativity: {A} simple proof for the lattice property
  of systems ordered by a semi-associative law.
\newblock {\em J. Combinatorial Theory Ser. A}, 13:7--13, 1972.

\bibitem[{Pon}]{JNotebook}
V.~{Pons}.
\newblock Live demo notebook with sage computation.
\newblock Github:
  \url{https://github.com/VivianePons/public-notebooks/tree/master/TamariIntervalPosets}.

\bibitem[PR12]{PRThesis}
L.-F. Préville-Ratelle.
\newblock {\em Combinatoire des espaces coinvariants trivariés du groupe
  symétrique}.
\newblock {T}h\`ese de {D}octorat, Université du Québec à Montréal, 2012.

\bibitem[PX15]{RatelleViennot}
L.-F. {Préville-Ratelle} and Viennot X.
\newblock {An extension of Tamari lattices}.
\newblock {\em DMTCS Proceedings, 27th International Conference on Formal Power
  Series and Algebraic Combinatorics}, 0(01), 2015.

\bibitem[{Rog}18]{PRE_RognerudModern}
B.~{Rognerud}.
\newblock {Exceptional and modern intervals of the Tamari lattice}.
\newblock {\em arXiv e-prints}, page arXiv:1801.04097, Jan 2018.

\bibitem[SCc08]{SAGE_COMBINAT}
The {S}age-{C}ombinat community.
\newblock {S}age-{C}ombinat: enhancing {S}age as a toolbox for computer
  exploration in algebraic combinatorics, 2008.
\newblock \url{http://combinat.sagemath.org}.

\bibitem[SD17]{SageMath2017}
The {S}age {D}evelopers.
\newblock {\em {S}ageMath, the {S}age {M}athematics {S}oftware {S}ystem
  ({V}ersion 7.5.1)}, 2017.
\newblock {\tt http://www.sagemath.org}.

\bibitem[Tam62]{Tamari1}
D.~Tamari.
\newblock The algebra of bracketings and their enumeration.
\newblock {\em Nieuw Arch. Wisk. (3)}, 10:131--146, 1962.

\end{thebibliography}
